\documentclass[english]{article}
\usepackage[T1]{fontenc}
\usepackage[latin9]{inputenc}
\usepackage{color}
\usepackage{amsmath}
\usepackage{amsthm}
\usepackage{amssymb}

\makeatletter
\numberwithin{equation}{section}
\numberwithin{figure}{section}
\theoremstyle{plain}
\newtheorem{thm}{\protect\theoremname}
\theoremstyle{plain}
\newtheorem{lem}[thm]{\protect\lemmaname}
\theoremstyle{definition}
\newtheorem{defn}[thm]{\protect\definitionname}
\theoremstyle{remark}
\newtheorem*{claim*}{\protect\claimname}
\theoremstyle{plain}
\newtheorem{cor}[thm]{\protect\corollaryname}
\theoremstyle{remark}
\newtheorem*{rem*}{\protect\remarkname}
\theoremstyle{remark}
\newtheorem{rem}[thm]{\protect\remarkname}
\theoremstyle{plain}
\newtheorem{prop}[thm]{\protect\propositionname}
\theoremstyle{definition}
\newtheorem{example}[thm]{\protect\examplename}

\usepackage{amssymb}
\usepackage{bbm}
\usepackage{amsmath}
\usepackage{amsthm}
\usepackage{tikz-cd}
\usepackage{hyperref}

\newcommand{\Pow}[1]{\mathcal {P}(#1)} 					
\newcommand{\M}{\mathcal {M}} 
\newcommand{\N}{\mathcal {N}} 
\newcommand{\W}{\mathcal {W}} 
\newcommand{\R}{\mathcal {R}}     					
\newcommand{\E}{\mathbb{E}} 						
 	
\newcommand{\fin}[1]{[#1]^{<\omega}}   				

\newcommand{\fu}[3]{ult_{#3}(#1,#2)}
\newcommand{\Hull}[2]{Hull_{#2}^{#1}}

\newcommand{\T}{\mathcal{T}}
\newcommand{\U}{\mathcal{U}}

\newcommand{\WO}{WO}

\newcommand{\cH}[2]{cHull_{#2}^{#1}}

\DeclareMathOperator{\crt}{crit}
\DeclareMathOperator{\Th}{Th}

\DeclareMathOperator{\lh}{lh}
\DeclareMathOperator{\wfp}{wfp}
\DeclareMathOperator{\ran}{ran}
\DeclareMathOperator{\pred}{pred}
\DeclareMathOperator{\lgcd}{lgcd}
\DeclareMathOperator{\cof}{cof}

\DeclareMathOperator{\OR}{OR}
\DeclareMathOperator{\dom}{dom}
\DeclareMathOperator{\otp}{otp}

\DeclareMathOperator{\col}{Col}
\DeclareMathOperator{\trc}{trc}

\DeclareMathOperator{\ad}{ad}
\DeclareMathOperator{\ck}{CK}

\makeatother

\usepackage{babel}
\providecommand{\claimname}{Claim}
\providecommand{\corollaryname}{Corollary}
\providecommand{\definitionname}{Definition}
\providecommand{\examplename}{Example}
\providecommand{\lemmaname}{Lemma}
\providecommand{\propositionname}{Proposition}
\providecommand{\remarkname}{Remark}
\providecommand{\theoremname}{Theorem}

\begin{document}
\title{On a Conjecture Regarding the Mouse Order for Weasels}
\author{Jan Kruschewski and Farmer Schlutzenberg\thanks{The first author was funded by the Deutsche Forschungsgemeinschaft
(DFG, German Research Foundation) grant no. 445387776.The second author
was funded by the Deutsche Forschungsgemeinschaft (DFG, German Research
Foundation) under Germany's Excellence Strategy EXC 2044 --390685587,
Mathematics Münster: Dynamics--Geometry--Structure.}}
\maketitle
\begin{abstract}
We investigate Steel's conjecture in 'The Core Model Iterability Problem'
\cite{CMIP}, that if $\W$ and $\R$ are $\Omega+1$-iterable, $1$-small
weasels, then $\W\leq^{*}\R$ iff there is a club $C\subset\Omega$
such that for all $\alpha\in C$, if $\alpha$ is regular, then $\alpha^{+\W}\leq\alpha^{+\R}$. 

We will show that the conjecture fails, assuming that there is an
iterable premouse which models $KP$ and which has a $\Sigma_{1}$-Woodin
cardinal. 

On the other hand, we show that assuming there is no transitive model
of $KP$ with a Woodin cardinal the conjecture holds. 

In the course of this we will also show that if $M$ is an iterable
admissible premouse with a largest, regular, uncountable cardinal
$\delta$, and $\mathbb{P}$ is a forcing poset such that $M\models"\mathbb{P}\text{ has the }\delta\text{-c.c.}"$,
and $g\subset\mathbb{P}$ is $M$-generic, $M[g]\models KP$. Moreover,
if $M$ is such a premouse and $\T$ is a maximal normal iteration
tree on $M$ such that $\T$ is non-dropping on its main branch, then
$\M_{\infty}^{\T}$ is again an iterable admissible premouse with
a largest regular and uncountable cardinal.

At last we will answer another question from \cite{CMIP} about the
$S$-hull property.
\end{abstract}

\section{Introduction}

In his book 'The Core Model Iterability Problem' \cite{CMIP} John
Steel conjectured on p.28 that if $\W$ and $\R$ be $1$-small weasels
which are $\Omega+1$-iterable, then $\W\leq^{*}\R$ iff there is
a club $C\subset\Omega$ such that for all $\alpha\in C$, if $\alpha$
is regular, then $(\alpha^{+})^{\W}\leq(\alpha^{+})^{\R}$.

In the terminology of the book a weasel is premouse of ordinal height
$\Omega$, where $\Omega$ is a fixed measurable cardinal. The relation
$\leq^{*}$ is the mouse order, i.e if $\W$ and $\R$ are weasels
which are sufficiently iterable to successfully coiterate (by results
of the book $\Omega+1$-iterablility suffices), then $\W\leq^{*}\R$
iff $\R$ wins the coiteration, i.e. if $(\T,\U)$ is the successful
coiteration of $(\W,\R)$ such that $\lh(T)=\theta+1$ and $\lh(\U)=\gamma+1$,
then $\M_{\theta}^{\T}\unlhd\M_{\gamma}^{\U}$.

Steel showed in \cite{CMIP} that the Conjecture holds for weasels
small enough that linear iterations suffice for comparison. His proof
is based on universal linear iterations. 

In the following we will prove the Conjecture under the assumption
that neither $\W$ nor $\R$ have an initial segment which models
the theory $KP$ and has a Woodin cardinal, see Theorem \ref{thm:Main Theorem of Section One}.
In particular, the Conjecture holds if there is no transitive model
of $KP$ with a Woodin cardinal. 

On the other hand, assuming the existence of an iterable admissible
premouse with a $\Sigma_{1}$-Woodin cardinal, we will construct a
counterexample to the Conjecture, see Section \ref{subsec:The-Construction-of the counterexample}.

For the construction of the counterexample from the assumption just
described we will need to investigate the extender algebra over admissible
mice. In the course of this, we will show that if $M$ is an iterable
admissible premouse with a largest, regular, uncountable cardinal
$\delta$, $\mathbb{P}$ is a forcing poset such that $M\models"\mathbb{P}\text{ has the }\delta\text{-c.c.}"$,
and $g\subset\mathbb{P}$ is $M$-generic, $M[g]\models KP$, see
Theorem \ref{thm:KP in delta c.c. forcings}

In Section \ref{sec:On-another-Conjecture} we will answer another
question from p. 32 of \cite{CMIP} about the $S$-hull property in
$1$-small weasels. We will show that for an $\Omega+1$-iterable
weasel $\W$ such that $\Omega$ is $S$-thick in $\W$ the set of
points that have the $S$-hull property is closed, see Theorem \ref{thm:Hull Property Set is Closed}.

We will use the notation from \cite{MR4290497}. A $k$-maximal iteration
tree is in our terminology a $k$-maximal iteration tree in the sense
of \cite{MR2768698}. 

\section{Where the Conjecture Holds}

In this section we will show under the assumption that there is no
transitive model of $KP$ with a Woodin cardinal that Steel's conjecture
holds. This will be a consequence of Theorem \ref{thm:Main Theorem of Section One}.
Lemma \ref{lem:Key Lemma of Section One} is the key insight for proving
Theorem \ref{thm:Main Theorem of Section One}. Before we prove this
Lemma let us recall some well-known basic properties of weasels and
their coiterations.

The following Lemma from \cite{CMIP} guarantees that the coiteration
of two $\Omega+1$-iterable weasels of height $\Omega$ is successful.
\begin{thm}
\label{thm:Coiteration of Weasels}Let $\kappa$ be an inaccessible
cardinal. Let $\M$ and $\N$ be premice such that $\OR^{\M}=\OR^{\N}=\kappa$,
and suppose that $\M$ and $\N$ are $\kappa+1$-iterable. Let $(\mathcal{T},\mathcal{U})$
be the terminal coiteration of $(\M,\N)$. Then $(\T,\U)$ is successful
and $\max\{\lh(\T)\lh(\U)\}\leq\kappa+1$. Moreover, setting $\lh(\T)=\theta+1$
and $\lh(\U)=\gamma+1$, either
\begin{enumerate}
\item $D^{\T}\cap[0,\theta]_{T}=\emptyset$, $i_{0,\theta}^{\T}[\kappa]\subset\kappa$,
$\M_{\theta}^{\T}\unlhd\M_{\gamma}^{\U}$, and $\OR^{\M_{\theta}^{\T}}=\kappa$
or 
\item $D^{\U}\cap[0,\gamma]_{U}=\emptyset$, $i_{0,\gamma}^{\U}[\kappa]\subset\kappa$,
$\M_{\gamma}^{\U}\unlhd\M_{\theta}^{\T}$, and $\OR^{\M_{\gamma}^{\U}}=\kappa$.
\end{enumerate}
\end{thm}

The next lemmas are basic facts about iteration trees whose proofs
are well-known.
\begin{lem}
\label{lem:weasel lemma for iterate of ordinal height kappa}Let $\kappa$
be a regular and uncountable cardinal and $k\leq\omega$. Let $\M$
be a premouse such that $\OR^{\M}=\kappa$ and let $\T$ be a $k$-maximal
iteration tree on $\M$ such that $\lh(\T)=\kappa+1$, there is no
dropping on $b:=[0,\kappa]_{T}$, and that $i_{0\kappa}^{\T}[\kappa]\subset\kappa$.
Then there is a club $C\subset\kappa$ such that for all $\alpha\in C$,
$i_{0\alpha}^{\T}[\alpha]\subset\alpha$.
\end{lem}

\begin{lem}
\label{lem:club of sups}Let $\kappa$ be a regular and uncountable
cardinal and $k\leq\omega$. Let $\M$ be a premouse such that $\OR^{\M}=\kappa$
and let $\T$ be a $k$-maximal iteration tree on $\M$ such that
$\lh(\T)=\kappa+1$ and there is no dropping on $b:=[0,\kappa]_{T}$.
Suppose that for all $\alpha<\kappa$, $\OR^{\M_{\alpha}^{\T}}\leq\kappa$.
Then there is a club $C\subset\kappa$ such that for all $\alpha\in C$,
$\alpha=\sup\{\lh(E_{\beta}^{\T}):\beta<\alpha\}$. 
\end{lem}

\begin{lem}
\label{lem:Enlarging iteration of weasels} Let $\kappa$ be a regular
and uncountable cardinal and $k\leq\omega$. Let $\M$ be a premouse
such that $\OR^{\M}=\kappa$ and let $\T$ be a $k$-maximal iteration
tree on $\M$ such that $\lh(\T)=\kappa+1$. Suppose that $\OR^{\M_{\alpha}^{\T}}=\kappa$
for all $\alpha<\kappa$ and $\OR^{\M_{\kappa}^{\T}}>\kappa$. Then
there is a club of $\alpha$'s such that $i_{\alpha\kappa}^{\T}(\alpha)=\kappa$
and $\crt(i_{\alpha\kappa}^{\T})=\alpha$.
\end{lem}

For the proof of Lemma \ref{lem:Key Lemma of Section One} we need
a version of Ville's Lemma. Let us first introduce some terminology.
\begin{defn}
Let $M=(\ulcorner M\urcorner,\in^{M},\E^{M},F^{M})$ be a $\mathcal{L}$-structure,
where $\mathcal{L}$ is the language of premice. We call 
\[
\wfp{M}:=\{x\in\ulcorner M\urcorner:\in^{M}\restriction(\trc_{\in^{M}}(x))^{2}\text{ is well-founded}\}
\]
 the well-founded part of $M$, where $\trc_{\in^{M}}(x)$ is the
closure of $\{x\}$ under $\in^{M}$.
\end{defn}

\begin{defn}
We write $"V=L[E]"$ as an abbreviation for the conjunction of the
Axiom of Extensionality, the Axiom of Foundation, and the $\mathcal{L}_{\in,\dot{E}}$-sentence
\[
\forall x\exists\alpha(x\in S_{\alpha}[\dot{E}])\land\forall\alpha\exists x(x\notin S_{\alpha}^{E}),
\]
where $S_{\alpha}[\dot{E}]$ is the refinement of $J[\dot{E}]$-hierachy
as described in \cite{MR3243739} on p. 80. Note that the second conjunct
makes sure that the model is an actual element of the $J$-hierarchy. 
\end{defn}

Note that for any $N$ such that $N\models V=L[E]$ and $N\models\exists\alpha(N\mid\alpha\models\psi)$
for some $\mathcal{L}$-formula $\psi$, there is a least such $\alpha$
and moreover, this can be described internally, as 
\[
N\models\forall\psi(\exists\alpha(N\mid\alpha\models\psi)\implies(\exists\beta(N\mid\beta\models\psi)\land\forall\gamma<\beta(N\mid\gamma\not\models\psi))).
\]

This follows essentially from the Axiom of Foundation. We will leave
this as an exercise for the reader. 

The next lemma is a slight variation of a lemma due to Ville
\begin{lem}
\label{lem:Villes Lemma}Let $M=(\ulcorner M\urcorner,\in^{M},\E^{M})$
be a $\mathcal{L}$-structure such that $M\models V=L[\dot{E}]$.
Suppose that the structure $M$ is not well-founded. Then $(\wfp(M),\in^{M}\restriction\wfp(M)^{2},\E^{M}\cap\wfp(M))\models KP$. 
\end{lem}

\begin{proof}
As $(\wfp(M),\in^{M}\restriction\wfp(M)^{2})$ satisfies the Axiom
of Extensionality, we may identify $(\wfp(M),\in^{M}\restriction\wfp(M)^{2},\E^{M}\cap\wfp(M))$
with its transitive collapse. The Axioms of Extensionality and Foundation
hold in $(\wfp(M),\in\restriction\wfp(M)^{2},\E^{M}\cap\wfp(M))$,
since $\wfp(M)$ is transitive.

Note that for $x\in M$ such that $\{z:z\in^{M}x\}\subset\wfp(M)$,
$x\in\wfp(M)$. Moreover, $V=L[E]$ implies the axiom of Pairing and
Union. Thus, we have that Pairing and Union hold in $(\wfp(M),\in\restriction\wfp(M)^{2},\E^{M}\cap\wfp(M))$. 

By the absoluteness of $\Sigma_{0}$-formulae, we have that for every
$\Sigma_{0}$-formula $\varphi$ in the language $\mathcal{L}$ and
$x_{1},...,x_{n}\in\wfp(M)$, 
\[
M\models\varphi(x_{1},...,x_{n})\iff(\wfp(M),\in\restriction\wfp(M)^{2},\E^{M}\cap\wfp(M))\models\varphi(x_{1},...,x_{n}).
\]
To show that $\Sigma_{0}$-comprehension holds in $(\wfp(M),\in\restriction\wfp(M)^{2},\E^{M}\cap\wfp(M))$,
note that $\Sigma_{0}$-comprehension holds in $M$, since $V=L[E]$.
Thus, for every $\Sigma_{0}$-formula and $x,p\in\wfp(M)$ we have
that $M\models\exists y\forall z(z\in y\leftrightarrow\varphi(z,p)\land z\in x)$.
But since $x\in\wfp(M)$, $y\subset\wfp(M)$ by our previous remark,
$y\in\wfp(M)$. By the absoluteness of $\Sigma_{0}$-formulae, $\Sigma_{0}$-comprehension
holds in $(\wfp(M),\in\restriction\wfp(M)^{2},\E^{M}\cap\wfp(M))$.

It remains to see that $\Sigma_{0}$-bounding holds in $(\wfp(M),\in\restriction\wfp(M)^{2},\E^{M}\cap\wfp(M))$.
Suppose that for $x,p\in\wfp(M)$ and a $\Sigma_{0}$-formula $\varphi$
it holds that $(\wfp(M),\in\restriction\wfp(M)^{2},\E^{M}\cap\wfp(M))\models\forall y\in x\exists z\varphi(z,y,p)$.
Since we assumed that $M$ is not well-founded, we have that $\wfp(M)\neq M$.
But then for every $\alpha\in\OR^{M}\setminus\OR^{\wfp(M)}$ it holds
that $M\models\forall y\in x\exists z\in J_{\alpha}^{M}\varphi(z,y,p)$
by the upwards absoluteness of $\Sigma_{1}$-formulae. However, since
$M\models V=L[E]$ there is a least such $\alpha\in\OR^{M}$ such
that $\alpha\in\wfp(M)$, since for all $\alpha\in\OR^{M}\setminus\OR^{\wfp(M)}$
it holds that $M\models\forall y\in x\exists z\in J_{\alpha}^{M}\varphi(z,y,p)$.
Thus, $(\wfp(M),\in\restriction\wfp(M)^{2},\E^{M}\cap\wfp(M))$ models
$\Sigma_{0}$-collection.
\end{proof}
We need the following version of the $\Sigma_{1}^{1}$-Bounding Theorem
in order to prove Lemma \ref{lem:Key Lemma of Section One}
\begin{thm}
\label{thm:Sigma One One Bounding}Let $x\in^{\omega}\omega$ and
suppose that $A\subset\WO$ is $\Sigma_{1}^{1}(x)$-definable, where
$\WO\subset^{\omega}\omega$ is the set of reals coding well-orderings.
Then 
\[
\sup\{<_{x}:x\in A\}<\omega_{1}^{\ck}(x).
\]
\end{thm}

The proof of Theorem \ref{thm:Sigma One One Bounding} is basically
a refinement of the Kunen-Martin Theorem, which may be found in \cite{dst}
on page 75. 

Note that it is a standard fact (or sometimes the definition) that
$\omega_{1}^{\ck}(x)$ is least (infinite) ordinal $\alpha$ such
that $L_{\alpha}(x)\models KP$.

The key lemma in proving Theorem \ref{thm:Main Theorem of Section One}
is the following
\begin{lem}
\label{lem:Key Lemma of Section One}Let $M$ be a premouse such that
$\delta$ is a regular, uncountable cardinal of $M$ and $\delta^{+}$
exists in $M$. Suppose that $M$ is $(0,\delta+1)$-iterable, and
there is no initial segment of $M$ which models $KP$ and has a Woodin
cardinal. Let $\U$ be a $0$-maximal iteration tree on $M$ of such
that
\begin{itemize}
\item $\lh(\U)=\delta+1$,
\item $\sup\{\lh(E_{\alpha}^{\U}):\alpha<\delta\}=\delta$, 
\item $b\cap D^{U}=\emptyset$, where $b:=[0,\delta]_{U}$, i.e. $\U$ is
non-dropping on its main branch, and
\item $i_{0\delta}^{\U}(\delta)=\delta$.
\end{itemize}
Then $i_{0\delta}^{\U}(\delta^{+M})=\delta^{+M}$.
\end{lem}

\begin{proof}
Suppose for the sake of contradiction that there is an iteration tree
$\U$ of length $\lh(\U)=\delta+1$ on $M$ such that $\sup\{\lh(E_{\alpha}^{\U}):\alpha<\delta\}=\delta$,
$b\cap D^{U}=\emptyset$, $i_{0\delta}^{\U}(\delta)=\delta$, and
$i_{0\delta}^{\U}(\delta^{+M})>\delta^{+M}$. Let $\xi<\delta^{+M}$
be such that $i_{0\delta}^{\U}(\xi)>\delta^{+M}$ and $\rho_{\omega}^{M\mid\xi}=\delta$.
Note that such $\xi$ exists since $i_{0\delta}^{\U}$ is continuous
at $\delta^{+M}$ , as $\delta^{+M}$ it is not of measurable cofinality
in any of the models of $\U$ and $\U$ is $0$-maximal. Moreover,
it is a standard fact that the ordinals $\eta$ such that $\rho_{\omega}^{M\mid\eta}=\delta$
are cofinal in $\delta^{+M}$. We also assume that $\xi>\OR^{Q}$,
where $Q\lhd M$ is least such that $J(Q)\models"\delta\text{ is not Woodin}"$.
Note that $Q$ exists and $\OR^{Q}<\delta^{+M}$, since there is no
initial segment of $M$ which models $KP$ and has $\delta$ as a
Woodin.

Let $g\subset\col(\omega,\delta)$ be $V$-generic. In particular,
$g$ is also $M$-generic. Let $x\in M[g]\cap^{\omega}\omega$ code
$M\mid\xi+\omega$ , i.e. $x$ codes $\in\cap(M\mid\xi+\omega)^{2}$
and the set of ordinals coding $\E^{M\mid\xi+\omega}$.

Let us define in $V[g]$ the set $\bar{A}$ such that $\T\in\bar{A}$
iff
\begin{enumerate}
\item $\T$ is a $0$-maximal normal putative iteration tree on $M\mid\xi+\omega$
such that $\lh(\T)=\delta+1$
\item $\T$ is non-dropping on its main branch,
\item $\sup\{\lh(E_{\alpha}^{\T}):\alpha<\delta\}=\delta$,
\item $i_{0\delta}^{\T}(\delta)=\delta$, thus in particular, $\M_{\delta}^{\T}\mid\delta\in\wfp(\M_{\delta}^{\T})$,
and
\item $\T\restriction\delta$ is guided by $Q$-structures.
\end{enumerate}
\begin{claim*}
If $\T\in\bar{A}$, then there is a $Q$-structure for $\T\restriction\delta$
in $V[g]$, i.e. there is $\alpha<\OR$ such that $J_{\alpha}(\M(\T\restriction\delta)\models"\delta=\delta(\T\restriction\delta)\text{ is not Woodin}"$.
\end{claim*}
\begin{proof}
Suppose that there is no $Q$-structure in $V[g]$. This means that
for all $\alpha<\OR$, $J_{\alpha}(\M(\T\restriction\delta))\models"\delta\text{ is Woodin}"$.
Let $\alpha$ be least such that $J_{\alpha}(\M(\T\restriction\delta))\models KP\land"\delta\text{ is Woodin}"$.
Since $\delta$ is regular in $L(\M(\T\restriction\delta))$, there
is $X\prec J_{\alpha}(\M(\T\restriction\delta))$ such that $X\cap\delta\in\delta$.
Let $N$ be the transitive collapse of $X$. Note that since $\M(\T\restriction\delta)=\bigcup_{\alpha\in[0,\delta)_{T}}\M_{\alpha}^{\T}\mid\lh(E_{\alpha}^{\T})$
and $[0,\delta]_{T}\cap D^{T}=\emptyset$, there is no initial segment
of $\M(\T\restriction\delta)$ which models $KP$ and has a Woodin
by elementarity. It follows that $N\lhd\M(\T\restriction\delta)$.
But $N$ models $KP$ and has a Woodin and for some $\alpha\in[0,\delta)_{T}$,
$N\lhd\M_{\alpha}^{\T}$.Contradiction!
\end{proof}
Since for $\T\in\bar{A}$ there is a $Q$-structure for $\T\restriction\delta$
and $\sup\{\lh(E_{\alpha}^{\T}):\alpha<\delta\}=\delta$ (and thus
$Q$-structures for $\T\restriction\alpha$ with $\alpha<\delta$
are of ordinal height less than $\delta$) it follows by the usual
absoluteness arguments that there is a unique cofinal  well-founded
branch $b$ for $\T\restriction\delta$ in $V[g]$\@.

However, at this point we do not know whether for $\T\in\bar{A}$,
the unique cofinal well-founded branch of $\T\restriction\delta$
was actually chosen, i.e. whether $\T$ is an iteration tree. This
is verified in the next Claim.
\begin{claim*}
If $\T\in\bar{A}$, then $\M_{\delta}^{\T}$ is well-founded. 
\end{claim*}
\begin{proof}
Suppose otherwise. Let $\T\in\bar{A}$ be such that $\M_{\delta}^{\T}$
is ill-founded. By 4. we have that $\M_{\delta}^{\T}\mid\delta\in\wfp(\M_{\delta}^{\T})$.
As we are assuming that $\M_{\delta}^{\T}$ is ill-founded, we may
apply Lemma \ref{lem:Villes Lemma} and get that $(\wfp(\M_{\delta}^{\T}),\in\restriction\wfp(\M_{\delta}^{\T})^{2},\E^{\M_{\theta}^{\T}}\cap\wfp(\M_{\delta}^{\T}))\models KP$.
To arrive at a contradiction we distinguish two cases. 

The first case is that $Q(c,\T\restriction\delta)\lhd\wfp(\M_{\delta}^{\T})$,
where $c=[0,\delta]_{T}$. Let $b$ be the unique, cofinal, and well-founded
branch $b$ for $\T\restriction\delta$ in $V[g]$\@. By $1$-smallness
we have that $Q(c,\T\restriction\delta)=Q(b,\T\restriction\delta)$.
But by the Zipper Lemma we have that $\delta$ is Woodin in $J(Q(b,\T\restriction\delta))$.
Contradiction!

The second case is that the $Q$-structure is not contained in $\wfp(\M_{\delta}^{\T})$.
In this case we have that for every proper initial segment of $\wfp(\M_{\delta}^{\T})$,
$\M_{\delta}^{\T}$ thinks that $\delta$ is Woodin in that segment.
But this means that $\delta$ is Woodin in $\wfp(\M_{\delta}^{\T})$.

Let $b$ be unique, cofinal, and well-founded branch $b$ for $\T\restriction\delta$
in $V[g]$. By $1$-smallness $Q(b,\T\restriction\delta)=J_{\alpha}(\M(\T\restriction\delta))$
for some $\alpha\in\OR$. Moreover, since we picked $\xi$ greater
than $\OR^{Q}$, we will have that $Q(b,\T\restriction\delta)\lhd\M_{b}^{\T\restriction\delta}$.

Note that by the elementarity of $i_{0\delta}^{\T}$, $\M_{\delta}^{\T}=\M_{c}^{\T\restriction\delta}$
thinks that there is no initial segment which models $KP$ and has
$\delta$ as a Woodin. Suppose for the sake of contradiction that
$\wfp(\M_{\delta}^{\T})\neq J_{\beta}(\M(\T\restriction\delta))$
for some $\beta\in\OR$. This means that there is some extender indexed
at some $\alpha\in[\delta,\OR^{\wfp(\M_{\delta}^{\T})})$. Let $\alpha$
be the least such and let $E$ be the extender indexed there. As $\M_{\delta}^{\T}$
is $1$-small and therefore $\wfp(\M_{\delta}^{\T})$ is $1$-small,
it is easy to see that $\crt(E)\geq\delta$. But this means that $\crt(E)^{+\M_{\delta}^{\T}\mid\alpha}\in\wfp(\M_{\delta}^{\T})$.
But this gives us a lot of models $KP$ below $\alpha$ which have
$\delta$ as a Woodin cardinal. Contradiction! Thus, $\wfp(\M_{\delta}^{\T})=J_{\beta}(\M(\T\restriction\delta))$
for some $\beta\in\OR$.

But as $J_{\beta}(\M(\T\restriction\delta))\models KP\land"\delta\text{ is Woodin}"$,
we must have that $\beta\leq\alpha$. Thus, $J_{\beta}(\M(\T\restriction\delta))\lhd\M_{b}^{\T\restriction\delta}$.
But this contradicts that $M$ does not have a an initial segment
which models $KP$ and has a Woodin!
\end{proof}
The set $\bar{A}$ gives rise to the set $A=\{\OR^{\M_{\infty}^{\T}}:\T\in\bar{A}\}$.
Note that $A$ is a set of countable ordinals in $V[g]$, since the
trees are countable in $V[g]$ and also $M\mid\xi+\omega$ is countable
in $V[g]$. Thus, we may code $A$ as a set of reals $A^{*}$ by setting
$A^{*}:=\{x\in\WO:\otp(R_{x})\in A\}$, where $\WO$ is the complete
$\Pi_{1}^{1}$-set of reals coding well-orders and $R_{x}$ is the
relation coded by $x$.
\begin{claim*}
$A^{*}$ is $\Sigma_{1}^{1}(x)$.
\end{claim*}
\begin{proof}
The main difficulty is in expressing in a $\Sigma_{1}^{1}(x)$-fahsion
that for $\alpha<\delta$, $\M_{\alpha}^{\T}$ is well-founded. Fix
some $\alpha<\delta$. Note that if $\alpha$ is a successor, then
$\M_{\alpha}^{\T}$ is always well-founded, since $M$ is iterable
and we may assume inductively that $\U\restriction\alpha$ is guided
by $Q$-structures. Suppose that $\alpha<\delta$ is a limit. Since
$\sup\{\lh(E_{\alpha}^{\T}):\alpha<\delta\}=\delta$, we have that
$\delta(\T\restriction\alpha)<\delta$. Note that $\lh(E_{\alpha}^{\T})>\delta(\T\restriction\alpha)$
and $\M_{\alpha}^{\T}$ is $1$-small below $\delta$. Thus, $Q(b,\T\restriction\alpha)$,
where $b=[0,\alpha]_{\T}$, must exist and $Q(b,\T\restriction\alpha)\lhd\M_{\alpha}^{\T}\mid\lh(E_{\alpha}^{\T})$.
Thus, the well-foundedness of $Q(b,\T\restriction\alpha)$ can be
expressed via the existence of an isomorphism between the ordinals
of $Q(b,\T\restriction\alpha)$ and an initial segment of the ordinals
of the model coded by $x$. This is $\Sigma_{1}^{1}(x)$.
\end{proof}
Let $\U^{*}$ be the tree $\U$ considered on $M\mid\xi+\omega$.
Note that since $\sup\{\lh(E_{\alpha}^{\U}):\alpha<\delta\}=\delta$
and $\delta$ is a regular cardinal of $M$ we have that $D^{U}=D^{U^{*}}$
and $\lh(\U)=\lh(\U^{*})$. Moreover, we also have that for all $\alpha<\delta$,
$\M_{\alpha}^{\U^{*}}\mid i_{0\alpha}^{\U^{*}}(\delta)=\M_{\alpha}^{\U}\mid i_{0\alpha}^{\U}(\delta)$.
\begin{claim*}
$i_{0\delta}^{\U}(\xi)=i_{0\delta}^{\U^{*}}(\xi)$.
\end{claim*}
\begin{proof}
Let $\kappa$ be the critical point of $\E_{\alpha}^{\T}$, where
$\alpha=\inf(b\setminus1)$. It suffices to see that every function
$f:\kappa\rightarrow\xi$ which is in $M$ is also in $M\mid\xi+\omega$.
To this end note that since $\rho_{\omega}^{M\mid\xi}=\delta$, there
is a surjection $g:\delta\rightarrow\xi$ such that $g\in M\mid\xi+\omega=J(M\mid\xi)$.
Define a function $h:\kappa\rightarrow\delta$ such that $\alpha<\kappa$,
$h(\alpha)=\min(g^{-1}(f(\alpha)))$. As $\delta$ is regular and
$\kappa<\delta$, $h$ is essentially a bounded subset of $\delta$
and thus, $h\in M\mid\xi$. But since $f$ is definable from $g$
and $h$ this means that $f\in M\mid\xi+\omega$.
\end{proof}
\begin{claim*}
$\U^{*}\in\bar{A}$.
\end{claim*}
\begin{proof}
It suffices to see that $\U\restriction\delta$ is guided by $Q$-structures.
As in the previous claim, for every $\alpha<\delta$, if $b_{\alpha}=[0,\alpha]_{\U}$,
then $Q(b_{\alpha},\U\restriction\alpha)\lhd\M_{\alpha}^{\U}\mid\lh(E_{\alpha}^{\U})=\M_{\alpha}^{\U^{*}}\mid\lh(E_{\alpha}^{\U^{*}})$.
Thus, $\U\restriction\delta$ is guided by $Q$-structures. 
\end{proof}
Let $\bar{A}^{M[g]}$ the interpretation of the set $\bar{A}$ in
$M[g]$. Note that although $\U^{*}\in\bar{A}$, it is not true that
$\U^{*}\in\bar{A}^{M[g]}$. However, since $x\in M[g]$ we have by
Theorem \ref{thm:Sigma One One Bounding} that $A$ and $A^{M[g]}$
are bounded by $\omega_{1}^{\ck}(x)=\omega_{1}^{\ck}(x)^{M[g]}$. 

Since $g$ has the $\delta^{+M}$-c.c. we have $\delta^{+\M}=\delta^{+\M[g]}=\omega_{1}^{\M[g]}>\omega_{1}^{\ck}(x)$.
But $\U^{*}\in A$ and we assumed that $i_{0\delta}^{\U}(\xi)=i_{0\delta}^{\U^{*}}(\xi)>\delta^{+\M}$.
Contradiction!
\end{proof}
We need a second lemma with a similar flavor
\begin{lem}
\label{lem:Second Key Lemma}Let $M$ be a premouse such that $\delta$
is a regular, uncountable cardinal of $M$ and $\delta^{+}$ exists
in $M$. Suppose that $M$ is $(0,\delta+1)$-iterable, and there
is no initial segment of $M$ which models $KP$ and has a Woodin
cardinal. Let $\U$ be a $0$-maximal iteration tree on $M$ such
that $\lh(\U)=\delta+1$ and $\sup\{\lh(E_{\alpha}^{\U}):\alpha<\delta\}=\delta$.
Let $b:=[0,\delta]_{U}$. Suppose that there is $\eta\in b\cap\delta\setminus\sup(D^{\U})$
such that $i_{\eta\delta}^{\U}(\gamma)=\delta$ for some $\gamma<\delta$.

Then $\delta^{+\M_{\delta}^{\U}}<\delta^{+M}$.
\end{lem}

\begin{proof}
The proof follows closely the proof of Lemma \ref{lem:Key Lemma of Section One}.
We modify the definition of $\bar{A}$ by omitting 2.) and modifying
4.) to ``there exists an $\eta\in b\cap\delta\setminus\sup(D^{\T})$
and a $\gamma<\delta$ such that $i_{\eta\delta}^{\T}(\gamma)=\delta$''.
Moreover, we alter 1.) by only requiring that $\T$ is a putative
iteration tree on $M\mid\delta$. 

In showing that for all $\T\in\bar{A}$, $\T$ is an iteration tree
we encounter a problem, we have not had before: Can it happen that
$\T$ drops in model to some $P$ with a largest cardinal $\eta$
such that $P\models"\eta\text{ is Woodin}"$ and then $P$ iterates
to $\M_{\delta}^{\T}$ in such a way that $\M_{\delta}^{\T}\models KP$?
If this could happen and the $Q$-structure for $\M_{c}^{\T\restriction\delta}$,
where $c$ is the correct branch, is $\M_{c}^{\T\restriction\delta}=\wfp(M_{\delta}^{\T})$,
then we cannot derive a contradiction as before. However, by Lemma
\ref{lem:General Preservation of KP} this is not possible and therefore,
we may argue as in the proof of Lemma \ref{lem:Key Lemma of Section One},
in order to see that for $\T\in\bar{A}$, $\T$ is an iteration tree.

Similar to before $\U^{*}$ will be the tree $\U$ considered on $M\mid\delta$.
Again, $\U^{*}\in\bar{A}$ and the tree- and dropping-structure of
$\U$ and $\U^{*}$ are the same. In particular $\M_{\alpha}^{\U^{*}}\unlhd\M_{\alpha}^{\U}$
for all $\alpha\leq\delta$. 

By the same argument as in the proof of Lemma \ref{lem:Key Lemma of Section One},
we will have that $\OR^{\M_{\delta}^{\U^{*}}}<\omega_{1}^{\ck}(x)<\delta^{+M}$
where $x$ codes $M\mid\delta$. We claim that this implies that $\delta^{+\M_{\delta}^{\U}}\leq\delta^{+M}$.
If $D^{\U}\cap b\neq\emptyset$, i.e. $\U$ is dropping on its main
branch, $\M_{\delta}^{\U}=\M_{\delta}^{\U^{*}}$, so that $\delta^{+\M_{\delta}^{\U}}\leq\OR^{\M_{\delta}^{\U}}<\delta^{+M}$.
So suppose that $D^{\U}\cap b=\emptyset$, equivalently that $D^{\U^{*}}\cap b=\emptyset$.
Then $i^{\U}:M\rightarrow\M_{\delta}^{\U}$ and $i^{\U^{*}}:M\mid\delta\rightarrow\M_{\delta}^{\U^{*}}$
exist. Let $\eta<\delta$ be least such that there is $\bar{\delta}<\delta$
such that $i_{\eta\delta}^{\U^{*}}(\bar{\delta})=\delta$. If $\eta=0$,
then since $\delta$ is regular in $M$ and $\sup\{\lh(E_{\alpha}^{\U}):\alpha<\delta\}=\delta,$$\delta^{+\M_{\delta}^{\U}}=i^{\U}(\bar{\delta}^{+M})\leq i^{\U}(\delta)=\sup i^{\U^{*}}[\delta]<\omega_{1}^{\ck}(x)<\delta^{+M}$.
So suppose that $\eta>0$. Note that $\bar{\delta}^{+\M_{\eta}^{\U^{*}}}$
must exist, since otherwise $\bar{\delta}$ is the greatest cardinal
of $\M_{\eta}^{\U^{*}}$ which would imply that $\bar{\delta}\in\ran(i_{0\eta}^{\U^{*}})$.
But then $\eta=0$, contradicting our assumption on $\eta$. Moreover,
since there is no dropping on the main branch, $i_{0\eta}^{\U^{*}}$
is cofinal in $\OR^{\M_{\eta}^{\U^{*}}}$. But then, $\delta^{+\M_{\delta}^{\U}}=i_{\eta\delta}^{\U}(\bar{\delta}^{+\M_{\eta}^{\U}})=i_{\eta\delta}^{\U^{*}}(\bar{\delta}^{+\M_{\eta}^{\U^{*}}})\leq i_{\eta\delta}^{\U^{*}}(i_{0\eta}^{\U^{*}}(\gamma))=i^{\U^{*}}(\gamma)<\omega_{1}^{\ck}(x)$
for some $\gamma<\delta$.
\end{proof}
We are now ready to prove the main theorem of this section
\begin{thm}
\label{thm:Main Theorem of Section One}Let $\Omega$ be a measurable
cardinal. Let $\W$ and $\R$ be premice such that $\OR^{\W}=\OR^{\R}=\Omega$
and suppose that both are $\Omega+1$-iterable. Suppose that neither
$\W$ nor $\R$ have an initial segment which models $KP\land\exists\delta("\delta\text{ is Woodin}")$.
Then the following are equivalent:
\begin{itemize}
\item $\W\leq^{*}\R$, 
\item there is a club $C\subset\Omega$ such that for all $\alpha\in C$,
if $\alpha$ is regular, then $(\alpha^{+})^{\W}\leq(\alpha^{+})^{\R}$,
and 
\item there is a stationary set $S\subset\Omega$ such that for all $\alpha\in S$,
$(\alpha^{+})^{\W}\leq(\alpha^{+})^{\R}$.
\end{itemize}
\end{thm}

\begin{proof}
Let us first suppose that $\W\leq^{*}\R$ and aim to show that there
is a club $C\subset\Omega$ such that for all $\alpha\in C$, if $\alpha$
is regular, then $(\alpha^{+})^{\W}\leq(\alpha^{+})^{\R}$. Let $(\T,\U)$
be the coiteration of $(\W,\R)$. By Theorem \ref{thm:Coiteration of Weasels},
the coiteration is successful. Since $\W\leq^{*}\R$, $\M_{\Omega}^{\T}\unlhd\M_{\Omega}^{\U}$
and $\T$ does not drop on its main branch. In particular, $i^{\T}:\W\rightarrow\M_{\Omega}^{\T}$
exists. We may assume by padding the trees if necessary that $\lh(\T)=\lh(\U)=\Omega+1$.
Note that by Theorem \ref{thm:Coiteration of Weasels} $\OR^{\M_{\Omega}^{\T}}=\Omega$.
Thus, by Lemma \ref{lem:weasel lemma for iterate of ordinal height kappa},
there is a club $D_{T}\subset\Omega$ such that for all $\alpha\in D_{T}$,
$i_{0\alpha}^{\T}[\alpha]\subset\alpha$. Let $\tilde{D}_{T}$ be
the club given by Lemma \ref{lem:club of sups} for $\T$. Set $C_{T}:=D_{T}\cap\tilde{D}_{T}$.Let
$\alpha\in C_{T}$. Then, if $\alpha$ is regular, $i_{0\alpha}^{\T}(\alpha)=\alpha$.
Thus, by Lemma \ref{lem:Key Lemma of Section One}, we have that $\alpha^{+\W}=\alpha^{+\M_{\alpha}^{\T}}$.
But since $\alpha\in\tilde{D}_{T}$, $\crt(i_{\alpha\Omega}^{\T})\geq\alpha$.
This implies by the usual argument for ultrapowers, that $\alpha^{+\M_{\alpha}^{\T}}=\alpha^{+\M_{\Omega}^{\T}}$.Since
$\M_{\Omega}^{\T}\unlhd\M_{\Omega}^{\U}$ and $\Omega$ is a cardinal,
we have that $\alpha^{+\M_{\Omega}^{\T}}=\alpha^{+\M_{\Omega}^{\U}}$
for all $\alpha\in C_{T}$. 

Case 1: $\OR^{\M_{\Omega}^{\U}}>\Omega$.Let $D_{U}\subset\Omega$
be the club given by Lemma \ref{lem:Enlarging iteration of weasels}
intersected with the club given by Lemma \ref{lem:club of sups}.
Then for all $\alpha,\beta\in D_{U}$, such that $\alpha<\beta$,
we have that $i_{\alpha\beta}^{\U}(\alpha)=\beta$. Thus, for every
$\alpha\in D_{T}^{\prime}$ which is not the least element of $D_{T}$
and is regular, we have by Lemma \ref{lem:Second Key Lemma}, that
$\alpha^{+\M_{\alpha}^{\U}}<\alpha^{+\R}$. Moreover, since $\crt(i_{\alpha\Omega}^{\U})\geq\alpha$
(note we are beyond the drops of $\U$), $\alpha^{+\M_{\alpha}^{\U}}=\alpha^{+\M_{\Omega}^{\U}}$.
Thus, $C:=D_{U}^{\prime}\cap C_{T}$ witnesses the claim.

Case 2: $\OR^{\M_{\Omega}^{\U}}=\Omega$. Then by the usual argument
$i_{0\Omega}^{\U}$ exists. In this case we can construct a club $C_{U}$
for $\U$ as $C_{T}$ was constructed for $\T$ and set $C:=C_{T}\cap C_{U}$.

Now suppose that there is a club $C\subset\Omega$ such that for all
$\alpha\in C$, if $\alpha$ is regular, then $\alpha^{+\W}\leq\alpha^{+\R}$.
We aim to show that $\W\leq^{*}\R$. Let $(\T,\U)$ be the coiteration
of $(\W,\R)$ which again by Theorem \ref{thm:Coiteration of Weasels}
is successful. Suppose for the sake of contradiction that $\M_{\infty}^{\T}\rhd\M_{\infty}^{\U}$.
By the same arguments as before we can construct a club $D\subset\Omega$
such that for all regular $\alpha\in D$, $\alpha^{+\W}>\alpha^{+\R}$.
Contradiction!

We leave the third equivalence as an exercise for the reader.
\end{proof}
\begin{cor}
Let $\Omega$ be a measurable cardinal. Let $\W$ and $\R$ be premice
such that $\OR^{\W}=\OR^{\R}=\Omega$ and suppose that both are $\Omega+1$-iterable.
Suppose that there is no transitive model of $KP\land\exists\delta("\delta\text{ is Woodin}")$.
Then the following are equivalent:
\begin{itemize}
\item $\W\leq^{*}\R$, 
\item there is a club $C\subset\Omega$ such that for all $\alpha\in C$,
if $\alpha$ regular, then $(\alpha^{+})^{\W}\leq(\alpha^{+})^{\R}$,
and 
\item there is a stationary set $S\subset\Omega$ such that for all $\alpha\in S$,
$(\alpha^{+})^{\W}\leq(\alpha^{+})^{\R}$.
\end{itemize}
\end{cor}

\section{The Counterexample}

In this section we construct a counterexample to Steel's conjecture
assuming large cardinals. In order to construct the counterexample
we need the following picture: We need an iterable premouse $M$ which
has an initial segment $N\lhd M$ such that $N\models KP$ and $N$
has largest cardinal $\kappa$ which is Woodin in $N$. Moreover,
$\kappa$ must be a measurable cardinal in the larger premouse $M$
and there must exist some $\Omega\in(\kappa^{+M},\OR^{M})$ which
is measurable in $M$.

We want to show that such $N$ exists if we assume that $M$ is the
least iterable sound premouse, which models $KP$ and has a largest
cardinal $\delta$ which is $\Sigma_{1}$-Woodin in $M$. Here $\Sigma_{1}$-Woodin
in $M$ means that for every $A\in\Sigma_{1}^{M}(M)$ there is some
$\kappa<\delta$ which is $<\delta,A$-reflecting in $M$.

In Subsection \ref{subsec:Basic-Properties-of-Admissible-Mice} we
will collect some basic Lemmas about admissible mice. In Subsection
\ref{subsec:The-Construction-of N} we will construct the above $N$
from the least iterable sound premouse $M$, which models $KP$ and
has a largest cardinal $\delta$ which is $\Sigma_{1}$-Woodin in
$M$. Note that since $M$ is the least such mouse, we will have that
$\kappa$, the largest cardinal of $N$, is not $\Sigma_{1}$-Woodin
in $N$. For the construction of the counterexample it will be important
that $N[g]$ can compute the order-type of a well-founded set $g$
which is added by the extender algebra forcing. In order for this
to work we will show that in general $N[g]\models KP$ if $g\subset\mathbb{P}\in N$
is $N$-generic, $N\models KP$, and $N\models"\mathbb{P}\text{ has the }\delta\text{-c.c.}"$.
This will be the content of Subsection \ref{subsec:The-Preservation-of KP}.
In Subsection \ref{subsec:The-Construction-of the counterexample}
we will construct the counterexample.

\subsection{Basic Properties of Admissible Mice\label{subsec:Basic-Properties-of-Admissible-Mice}}
\begin{defn}
A premouse $M=(\ulcorner M\urcorner,\in,\E^{M},F^{M})$ is called
an admissible premouse if it is passive, i.e. $F^{M}=\emptyset$,
and $M\models KP$. 
\end{defn}

\begin{rem*}
Note that an active premouse cannot model $KP$. We leave this as
an exercise. 

Moreover, since admissible premice are passive, when taking fine-structural
ultrapowers of an admissible premouse we may work with $M$ directly
and not the ``zero-core'' $\mathcal{C}_{0}(M)$.

For $n\leq\omega$, we say a premouse $M$ is $n$-countably iterable
if every countable elementary substructure of $M$ is $(n,\omega_{1},\omega_{1}+1)^{*}$-iterable.
If $M$ is $\omega$-countably iterable we also say that $M$ is countably
iterable.

Note that for an admissible premouse $M$ $0$-countable iterability
and $1$-countable iterability are equivalent, since $\fu ME1=\fu ME0$
as for every function $f\in\Sigma_{1}^{M}(M)$ such that $\dom(f)\in M$,
we have by $\Sigma_{1}$-bounding, $f\in M$.
\end{rem*}
The following is a nice criterion for the admissibility of a passive
premice, whose proof we leave to the reader.
\begin{lem}
\label{lem:KP in Premice}Let $M$ be a premouse. Then $M\models KP$
iff for all $f\in\Sigma_{1}^{M}(M)$ such that $f$ is a function
with $\dom(f)\in M$, $f\in M$.

Moreover, if $M$ has a largest cardinal $\delta$, then $M\models KP$
iff for all $f\in\Sigma_{1}^{M}(M)$ such that $f$ is a function
with $\dom(f)=\delta$, $f\in M$.
\end{lem}

\begin{rem*}
If $M$ is a $0$-countably iterable premouse without a largest cardinal,
then $M$ is admissible. The proof of this uses the Condensation Lemma.
\end{rem*}
The following lemma is a standard result, whose proof we omit.
\begin{lem}
\label{lem:Largest Cardinal Lemma}Suppose that $M$ and $N$ are
premice and let $i:M\rightarrow N$ be cofinal and $r\Sigma_{1}$-elemenatary.
Then 
\[
M\models"\delta\text{ is a cardinal}"\iff N\models"i(\delta)\text{ is a cardinal}".
\]
Moreover, if 
\[
M\models"\delta\text{ is the largest cardinal}",
\]
then 
\[
N\models"i(\delta)\text{ is the largest cardinal}".
\]
If $\delta$ is $M$-regular, then $i(\delta)$ is $N$-regular.
\end{lem}

\begin{lem}
\label{lem:Projectum of KP Structures}Suppose $M$ is a $0$-countably
iterable admissible premouse. If $\rho_{1}^{M}<\OR^{M}$, then for
any $M$-cardinal $\delta$, $\rho_{1}^{M}\geq\delta$. In particular,
if $\rho_{1}^{M}<\OR^{M}$, then $\rho_{1}^{M}$ is the largest cardinal
of $M$.
\end{lem}

\begin{proof}
Let $\rho:=\rho_{1}^{M}$. Assume for the sake of contradiction that
$\rho^{+M}$ exists, i.e. $\rho^{+M}<\OR^{M}$. 
\begin{claim*}
$M\mid\rho^{+M}\subset\Hull M1(\rho\cup\{p_{1}^{M},\rho\})$
\end{claim*}
\begin{proof}
Note that we are not assuming that $M$ is $1$-sound. But since $M$
is iterable, we have that $p_{1}^{M}$ is $1$-universal, i.e. $\mathcal{P}(\rho)\cap M\subset\cH M1(\rho\cup\{p_{1}^{M}\})$,
where $\cH M1(\rho\cup\{p_{1}^{M}\})$ is the transitive collapse
of $\Hull M1(\rho\cup\{p_{1}^{M},\rho\})$. However, since we added
$\rho$ as an element of $\Hull M1(\rho\cup\{p_{1}^{M},\rho\})$,
we actually have that $\mathcal{P}(\rho)\cap M\subset\Hull M1(\rho\cup\{p_{1}^{M},\rho\})$.
\end{proof}
Note that by the upwards absoluteness of $\Sigma_{1}$-formulas the
claim implies 
\[
M\models\forall\beta<\rho^{+M}\exists\alpha(\beta\in\Hull{M\mid\alpha}1(\rho\cup\{p_{1}^{M},\rho\})).
\]
Note that the part of the formula in parentheses is $\Sigma_{1}$(as
the $\Sigma_{1}$-Skolem-Hull is $\Sigma_{1}$-definable). Thus, by
$\Sigma_{1}$-bounding there is some $\gamma<\OR^{M}$ such that 
\[
M\models\forall\beta<\rho^{+M}(\beta\in\Hull{M\mid\gamma}1(\rho\cup\{p_{1}^{M},\rho\})).
\]
But then there is $f\in J(M\mid\gamma)\unlhd M$ such that $f:\rho\rightarrow\rho^{+M}$
is onto. Contradiction! Thus, $\rho$ is the largest cardinal of $M$.
\end{proof}
\begin{cor}
\label{cor:Iterability implies soundness}If $M$ is a $0$-countably
iterable admissible premouse, then $M$ is $1$-sound.
\end{cor}

\begin{proof}
If $\rho_{1}^{M}=\OR^{M}$, then $M$ is trivially $1$-sound. If
$\rho_{1}^{M}<\OR^{M}$, then by the previous Lemma $\rho_{1}^{M}$
is the largest cardinal of $M$. But then 
\[
H:=\Hull M1(\rho_{1}^{M}\cup\{p_{1}^{M}\})
\]
is cofinal in $\rho_{1}^{M}=\OR^{M}$. Thus, since $\rho_{0}^{M}\in H$
(as it is the definable as the largest cardinal), $\rho+1\in H$ and
therefore $H$ contains all ordinals of $M$.
\end{proof}
The proof of Lemma \ref{lem:Projectum of KP Structures} also gives
the following version of Lemma \ref{lem:Projectum of KP Structures}
which replaces the iterability assumption by $1$-soundness.
\begin{lem}
\label{lem:1 sound largest cardinal}Suppose $M$ is an admissible
premouse, which is $1$-sound. If $\rho_{1}^{M}<\OR^{M}$, then for
any $M$-cardinal $\delta$, $\rho_{1}^{M}\geq\delta$. In particular,
$\rho_{1}^{M}$ is the largest cardinal of $M$.
\end{lem}

\subsection{The Construction of $N$\label{subsec:The-Construction-of N}}
\begin{defn}
Let $M_{\Sigma_{1}-1}^{\ad}$ be the least $(0,\omega_{1}+1)$-iterable
and sound premouse which models $KP\land\exists\delta(\delta\text{ is }\Sigma_{1}\text{-Woodin})$.
\end{defn}

Note that $\rho_{2}^{M_{\Sigma_{1}-1}^{\ad}}=\omega$, while by Lemma
\ref{lem:Largest Cardinal Lemma}, $\rho_{1}^{M_{\Sigma_{1}-1}^{\ad}}=\lgcd(M_{\Sigma_{1}-1}^{\ad})$.
Moreover, by our previous remark $(0,\omega_{1}+1)$-iterability is
equivalent to $(1,\omega_{1}+1)$-iterability for $M_{\Sigma_{1}-1}^{\ad}$.

\begin{defn}
Let $M$ be an admissible premouse such that $\rho_{1}^{M}<\OR^{M}$.
We let $T^{M}\subset\rho_{1}^{M}$ be the set of ordinals coding $\Th_{\Sigma_{1}}^{M}(\rho_{1}^{M}\cup\{p_{1}^{M}\})$.
\end{defn}

Note that by Lemma \ref{lem:Projectum of KP Structures} $\rho_{1}^{M}$
is the largest cardinal of $M$. Moreover, by definition $T^{M}\notin M$,
but $T^{M}\in\Sigma_{1}^{M}(M)$. However, for every $\alpha<\rho_{1}^{M}$,
$T^{M}\cap\alpha\in M$. Thus, when $\pi$ is an embedding with domain
$M$ which is continuous at $\rho_{1}^{M}$, we write $\pi(T)$ for
$\bigcup_{\alpha<\rho_{1}^{M}}\pi(T^{M}\cap\alpha)$. 
\begin{lem}
Let $M$ and $N$ be admissible $1$-sound premice which both project,
i.e. $\rho_{1}^{M}<\OR^{M}$ and $\rho_{1}^{N}<\OR^{N}$. Let $\pi:M\rightarrow N$
be a $1$-elementary embedding and suppose that $\pi$ is continuous
at $\rho_{1}^{M}$. Then $\pi(T^{M})=T^{N}$.
\end{lem}

\begin{proof}
A $r\Sigma_{2}$-formula is of the form $\exists a\exists b(\dot{T}_{1}(a,b)\land\psi(a,b,\vec{v}))$,
where $\dot{T}_{1}^{M}(a,b)$ holds iff $a=\langle\alpha,q\rangle$
for some $\alpha<\rho_{1}^{M}$and $q\in M$, and $b=\Th_{1}^{M}(\alpha\cup\{q\})$.
Since $\pi(\rho_{1}^{M})=\rho_{1}^{N}$, the $r\Sigma_{2}$-elementarity
of $\pi$ implies that $\pi(T^{M})\subset T^{N}$. On the other hand,
since $\pi\restriction\rho_{1}^{M}$ is cofinal in $\pi(\rho_{1}^{N})$,
$\pi(T^{N})\subset T^{M}$.
\end{proof}
\begin{lem}
In $M_{\Sigma_{1}-1}^{\ad}$, let $\kappa$ be $<\delta,T^{M_{\Sigma_{1}-1}^{\ad}}$-reflecting.
Let $\delta$ be the largest cardinal of $M_{\Sigma_{1}-1}^{\ad}$
and let $H:=\Hull M1(\kappa\cup\{\delta\})$. Then,
\begin{itemize}
\item $H\cap\delta=\kappa$, and
\item if $\pi:N\cong H\prec_{\Sigma_{1}}M_{\Sigma_{1}-1}^{\ad}$ is the
transitive collapse of $H$, then $N\models KP$.
\end{itemize}
\end{lem}

\begin{proof}
Let us first proof that $H\cap\delta=\kappa$. Suppose for the sake
of contradiction that there is $\xi\in H\cap[\kappa,\delta)$. Let
$\varphi_{\xi}$ be a $\Sigma_{1}$-formula and $\alpha_{1},...,\alpha_{n}\in\kappa$
be such that $\xi$ is the unique $\eta$ such that $M\models\varphi_{\xi}(\xi,\alpha_{1},...,\alpha_{n},\delta)$.
Note that the $\varphi_{\xi}(\xi,\alpha_{1},...,\alpha_{n},\delta)\in\Th_{1}^{M}(\delta\cup\{\delta\})$.
Thus, there is some $\lambda<\delta$ such that the code for $\varphi_{\xi}(\xi,\alpha_{1},...,\alpha_{n},\delta)$
is below $\lambda$, i.e. $\varphi_{\xi}(\xi,\alpha_{1},...,\alpha_{n},\delta)\in T^{M}\cap\lambda$.
Let $E\in\E^{M}$ be an extender witnessing that $\kappa$ is $\lambda,T^{M}$-strong.
Let $i_{E}:M\rightarrow\fu ME1$ be the canonical ultrapower embedding.
Note that we are taking a $1$-ultrapower, since $\rho_{2}^{M}=\omega<\kappa=\crt(E)<\delta=\rho_{1}^{M}$.
In particular, $i_{E}$ is a $r\Sigma_{2}$-elementary embedding.
By the previous Lemma this means that $i_{E}(T^{M})=T^{\fu ME1}$.
But since $E$ is $\lambda,T^{M}$-strong, we have $T^{\fu ME1}\cap\lambda=i_{E}(T^{M})\cap\lambda=T^{M}\cap\lambda$.
Thus, $\varphi_{\xi}(\xi,\alpha_{1},...,\alpha_{n},\delta)\in T^{\fu ME1}\cap\lambda$
which means that $\fu ME1\models\varphi_{\xi}(\xi,\alpha_{1},...,\alpha_{n},\delta)$.
Moreover, since $M\models\varphi_{\xi}(\xi,\alpha_{1},...,\alpha_{n},\delta)$
also by $r\Sigma_{2}$-elementarity, $\fu ME1\models\varphi_{\xi}(i_{E}(\xi),\alpha_{1},...,\alpha_{n},\delta)$.
But since $\xi$ was the unique witness for $\exists x\varphi_{\xi}(x,\alpha_{1},...,\alpha_{n},\delta)$
in $M$, we have 
\[
M\models\forall x(x\neq\xi\rightarrow\neg\varphi_{\xi}(x,\alpha_{1},...,\alpha_{n},\delta))
\]
but this is $r\Pi_{1}$and thus, 
\[
\fu ME1\models\forall x(x\neq i_{E}(\xi)\rightarrow\neg\varphi_{\xi}(x,\alpha_{1},...,\alpha_{n},\delta)).
\]
But this means $i_{E}(\xi)=\xi$. Contradiction!

Let us now verify that $N\models KP$. By what we have shown so far
$\pi(\kappa)=\delta$. In particular, by Lemma \ref{lem:Largest Cardinal Lemma}
$\kappa$ is the largest cardinal of $N$. So by Lemma \ref{lem:KP in Premice}
it suffices to check $\Sigma_{1}$-bounding. Suppose that 
\[
N\models\forall\alpha<\kappa\exists y\varphi(\alpha,y,\bar{p})
\]
where $\varphi$ is a $\Sigma_{1}$-formula and $\bar{p}\in N$. Let
$p=\pi(\bar{p})$. Then 
\[
M\models\forall\alpha<\kappa\exists y\varphi(\alpha,y,p).
\]
We aim to see that 
\[
M\models\forall\alpha<\delta\exists y\varphi(\alpha,y,p).
\]
Suppose for the sake of contradiction that there is some $\xi\in[\kappa,\delta)$
such that $M\models\forall y\neg\varphi(\xi,y,p)$. This means that
$\exists y\varphi(\xi,y,p)\notin T^{M}$. But then there is some $\lambda\in(\xi,\delta)$
such that this is witnessed by $T^{M}\cap\lambda$, i.e. the formula
$\exists y\varphi(\xi,y,p)\notin T^{M}\cap\lambda$ even though $\xi,p\in M\mid\lambda$.
(Note that this is not literally true as $p\in M\setminus M\mid\delta$
is possible. However, since $M$ is $1$-sound and $\rho_{1}^{M}=\delta$
and $p_{1}^{M}=\delta$ we may think of $p$ as a tuple $(\vec{\alpha},\delta)$,
where $\vec{\alpha}\in\fin\delta$. Since we coded $T^{M}$ as a theory
with constant symbol $\dot{\delta}$ it suffices if we pick $\lambda>\sup\{\xi,\vec{\alpha}\}$.)
Let $E\in\E^{M}$ be $\lambda,T^{M}$-strong. Then as in the previous
argument, $\exists y\varphi(\xi,y,i_{E}(p))"\notin"T^{\fu ME1}\cap\lambda$.
However, since $i_{E}$ is $r\Sigma_{2}$-elementary, we have that
\[
\fu ME1\models\forall\alpha<i_{E}(\kappa)\exists y\varphi(\alpha,y,p).
\]
Thus, since $i_{E}(\kappa)>\xi$, 
\[
\fu ME1\models\exists y\varphi(\xi,y,i_{E}(p)).
\]
Contradiction!
\end{proof}
\begin{lem}
\label{lem:N existence lemma}Let $N$ be as in the previous Lemma.
Then $N\lhd M$ and $N$ is an admissible premouse with largest cardinal
$\kappa$ which is Woodin in $N$.
\end{lem}

\begin{proof}
Note that a failure of $\kappa$ being Woodin in $N$ is a $r\Sigma_{1}$-fact
about $N\mid\kappa$. Since $H\prec_{\Sigma_{1}}M$, this would imply
by upwards-absoluteness that $\delta$ fails to be Woodin in $M$.
Thus, $\kappa$ is Woodin in $N$. Moreover, as already mentioned
in the previous proof $\kappa$ is the largest cardinal of $N$. But
then as $N\mid\kappa=M\mid\kappa$ and so $N=J_{\alpha}(M\mid\kappa)$,
we will have that $N\lhd M$ and we are done.
\end{proof}

\subsection{The Preservation of KP under the Extender Algebra Forcing\label{subsec:The-Preservation-of KP}}

In this subsection we will prove that if $M$ is an admissible mouse
with a largest cardinal $\delta$ which is Woodin in $M$, $\mathbb{B}$
is the extender algebra as defined inside $M$, and $g\subset\mathbb{B}$
is $M$-generic, then $M[g]\models KP$. This will be Corollary \ref{cor:KP preservations extender algebra}
which is an instance of the more general Theorem \ref{thm:KP in delta c.c. forcings}. 

Note that if $M$ is an admissible premouse with a largest, regular,
and uncountable cardinal $\delta$ and $\mathbb{P}\in M$ is a forcing
poset, then we may assume without loss of generality that $\mathbb{P}\subset\delta$,
as there is a surjection $f:\delta\rightarrow\mathbb{P}$ in $M$.
We will do so throughout without further mentioning this.
\begin{lem}
\label{lem:BoundedAntichains}Let $M$ be an admissible premouse with
a largest, regular, and uncountable cardinal $\delta$ and let $\mathbb{P}\in M$
be a forcing poset such that $M\models"\mathbb{P}\text{ has the }\delta\text{-c.c.}"$.
Then there is no $A\in\Sigma_{1}^{M}(M)$ such that $A\subset\mathbb{P}$
is an antichain in $\mathbb{P}$ which is unbounded in $\delta$.
\end{lem}

\begin{proof}
Suppose there is $A\in\Sigma_{1}^{M}(M)$ is such that $A\subset\mathbb{P}$
is an antichain in $\mathbb{P}$ which is unbounded in $\delta$.
Since $\mathbb{P}$ has the $\delta$-c.c. in $M$, and $\delta$
is regular in $M$, this means that $A\notin M$.

Let $\varphi$ be a $\Sigma_{1}$-formula such that for some $p\in M$,
\[
x\in A\iff M\models\varphi(x,p).
\]
Define for $\alpha<\OR^{M}\setminus\delta$ such that $p\in M\mid\alpha$,
\[
A_{\alpha}:=\{\xi\in\delta:M\mid\alpha\models\varphi(\xi,p)\}.
\]
Note that $A=\bigcup_{\alpha<\OR^{M}}A_{\alpha}$ and $A_{\alpha}\in M$
for all $\alpha$ such that $\alpha\in\OR^{M}\setminus\delta$ and
$p\in M\mid\alpha$. In particular, since $A_{\alpha}\subset A$ is
an antichain, $\mathbb{P}$ has the $\delta$-c.c. in $M$, and $\delta$
is regular in $M$, $A_{\alpha}$ is bounded in $\delta$ for all
such $\alpha$.

Note that since $A=\bigcup_{\alpha<\OR^{M}}A_{\alpha}$ is by assumption
unbounded in $\delta$,
\[
M\models\forall\beta<\delta\exists\alpha(\exists\gamma(\gamma>\beta\land\gamma\in A_{\alpha}).
\]
Since the last part of this formula is $\Sigma_{1}$, we have by $\Sigma_{1}$-collection
that there is some $\alpha^{\prime}\in M$ such that $\alpha^{\prime}$
works for all $\beta<\delta$ uniformly, i.e. 
\[
M\models\exists\alpha^{\prime}\forall\beta<\delta(\exists\gamma(\gamma>\beta\land\gamma\in A_{\alpha^{\prime}}).
\]
Thus, $A_{\alpha^{\prime}}$ is unbounded in $\delta$ and $A_{\alpha^{\prime}}\in M$.
But, since $A_{\alpha^{\prime}}\subset A$, $A_{\alpha^{\prime}}$
is an antichain in $\mathbb{P}$ and thus by the $\delta$-c.c. in
$M$, $M\models\mid A_{\alpha^{\prime}}\mid<\delta$. However, this
contradicts the regularity of $\delta$! Thus, $A$ cannot exist.
\end{proof}
\begin{lem}
\label{lem:Sigma1Regularity}Suppose that $M$ is an admissible premouse
with a largest, regular, and uncountable cardinal $\delta$. Then
$\delta$ is a $\Sigma_{1}^{M}(M)$-regular cardinal, i.e. for all
$\eta<\delta$ and $f\in\Sigma_{1}^{M}(M)$ such that $f:\eta\rightarrow\delta$,
$\ran(f)$ is bounded in $\delta$, i.e. $\sup(\ran(f))<\delta$.
\end{lem}

The proof of the Lemma is very similar to the proof of Lemma \ref{lem:BoundedAntichains}.
Thus, we leave it as an exercise for the reader.
\begin{defn}
\label{def:set of strongly forcing coniditions}Let $M$ be an admissible
premouse, $\mathbb{P}\in M$ be a forcing poset, and $g\subset\mathbb{P}$
be $M$-generic.

For $\beta<\OR^{M}$ such that $\mathbb{P}\in M\mid\beta$ it is a
standard fact about admissible structures, that $\Vdash_{M\mid\beta}^{\mathbb{P}}\in M$,
where $\Vdash_{M\mid\beta}^{\mathbb{P}}$ is the syntactical forcing
relation with $M\mid\beta$ as a ground model. Moreover, the function
$F:\OR^{M}\rightarrow M$ such that $F(\beta)=\Vdash_{M\mid\beta}^{\mathbb{P}}$
is $\Delta_{1}^{M}(\{\mathbb{P}\})$.

Let $\varphi$ be a formula and $\dot{z}_{1},...,\dot{z}_{n}\in M^{\mathbb{P}}$.
When we write $p\Vdash_{M\mid\beta}^{\mathbb{P}}\psi(\dot{x},\dot{z}_{1},...,\dot{z}_{n})$
for some $p\in\mathbb{P}$, we always implicitly assume that $\{\mathbb{P},\dot{z}_{1},...,\dot{z}_{n}\}\subset M\mid\beta$.

Suppose that $g\subset\mathbb{P}$ is $M$-generic and that for a
$\Sigma_{0}$-formula $\psi$, $M[g]\models\exists x\psi(x,z_{1},...,z_{n})$,
where $\{z_{1},...,z_{n}\}\subset M[g]$. Let $\varphi\equiv\exists x\psi$
be the corresponding $\Sigma_{1}$-formula and let $\dot{z}_{1},...,\dot{z}_{n}$
be $\mathbb{P}$-names for $z_{1},...,z_{n}$. We let 
\[
W_{\varphi,\dot{z}_{1},...,\dot{z}_{n}}^{M}:=\{p\in\mathbb{P}:\exists\beta<\OR^{M}\exists\dot{x}\in M^{\mathbb{P}}(p\Vdash_{M\mid\beta}^{\mathbb{P}}\psi(\dot{x},\dot{z}_{1},...,\dot{z}_{n})\}.
\]
We call the set $W_{\varphi,\dot{z}_{1},...,\dot{z}_{n}}^{M}$ the
set of conditions strongly forcing $\varphi$ (with parameters $\dot{z}_{1},...,\dot{z}_{n}$).
\end{defn}

\begin{rem*}
Since $M\models"\text{Pairing}"$ it follows easily (using $\mathbb{P}$-names)
that $M[g]\models"\text{Pairing}"$. Therefore we may arrange that
$n$ has any specific value which we want it to have. In particular,
$1$.

We will write $W_{\varphi,\dot{z}_{1},...,\dot{z}_{n}}$ for $W_{\varphi,\dot{z}_{1},...,\dot{z}_{n}}^{M}$
if it is clear from the context what $M$ is.

Since $F\in\Delta_{1}^{M}(\{\mathbb{P}\})$, it follows easily that
$W_{\varphi,\dot{z}}\in\Sigma_{1}^{M}(\{\dot{z},\varphi,\mathbb{P}\})$.
Moreover, since $\mathbb{P}\in M$, there is $\beta<\OR^{M}$ such
that $\mathbb{P}\in M\mid\beta$ so that $W_{\varphi,\dot{z}}$ is
not trivially empty. However, in general $W_{\varphi,\dot{z}}\notin M$.

For $p\in W_{\varphi,\dot{z}}$, we write $p\Vdash_{M}^{w.\mathbb{P}}\varphi(\dot{z})$
as an abbreviation for $\exists\beta<\OR^{M}\exists\dot{x}\in M^{\mathbb{P}}(p\Vdash_{M\mid\beta}^{\mathbb{P}}\psi(\dot{x},\dot{z})$.
\end{rem*}
We have the following forcing theorem for $\Sigma_{1}$-statements
for models of $KP$
\begin{thm}
\label{thm:Sigma 1 Forcing Theorem}Suppose that $M$ is an admissible
premouse. Let $\varphi(y_{1},...,y_{n})\equiv\exists x\psi(x,y_{1},...,y_{n})$,
where $\psi$ is a $\Sigma_{0}$-formula. Let $\mathbb{P}\in M$ be
a forcing poset, $\{\dot{z}_{1},...,\dot{z}_{n}\}\subset M^{\mathbb{P}}$,
and suppose that $g\subset\mathbb{P}$ be $M$-generic. Then the following
are equivalent:
\begin{itemize}
\item $A[g]\models\varphi(\dot{z}_{1}^{g},...,\dot{z}_{n}^{g})$,
\item there is $p\in W_{\varphi,\dot{z}_{1},...,\dot{z}_{n}}\cap g$, and
\item there is $p\in g$ such that $p\Vdash_{M}^{\mathbb{P}}\varphi(\dot{z}_{1},...,\dot{z}_{n})$.
\end{itemize}
\end{thm}

\begin{rem*}
Here $p\Vdash_{M}^{\mathbb{P}}\varphi(\dot{z}_{1},...,\dot{z}_{n})$
refers to the classical notion of forcing an existential statement,
i.e. for every $q\leq p$ there is some $r\leq q$ such that there
is $\dot{x}\in M^{\mathbb{P}}$ such that $r\Vdash_{M}^{\mathbb{P}}\psi(\dot{x},\dot{z}_{1},...,\dot{z}_{n})$.
In order to show that this is equivalent to the existence of $p\in W_{\varphi,\dot{z}_{1},...,\dot{z}_{n}}\cap g$,
we must use the admissibility of $M$. The rest of the proof is standard.
\end{rem*}
\begin{lem}
\label{lem:maximal antichain lemma}Let $M$ be a $0$-countably iterable
admissible premouse with a largest, regular, and uncountable cardinal
$\delta$. Let $\mathbb{P}\in M$ be a forcing poset such that $M\models"\mathbb{P}\text{ has the }\delta\text{-c.c.}"$.
Let $g\subset\mathbb{P}$ be $M$-generic. Suppose that for $\lambda<\delta$,
$M[g]\models\forall\alpha<\lambda\exists x\psi(x,\alpha,z)$, where
$\psi$ is a $\Sigma_{0}$-formula and $z\in M[g]$. Let $\varphi\equiv\exists x\psi$
and $\dot{z}\in M^{\mathbb{P}}$ be a name for $z$.

Then there is $A\in M$ such that $A\subset\delta\times\delta$ and
if for $\alpha<\lambda$, $A_{\alpha}:=\{\beta:(\alpha,\beta)\in A\}$,
then $A_{\alpha}\subset W_{\varphi,\check{\alpha},\dot{z}}$ is a
maximal antichain in $W_{\varphi,\check{\alpha},\dot{z}}$, i.e. for
every $p\in W_{\varphi,\check{\alpha},\dot{z}}$ there is some $q\in A_{\alpha}$
such that $q\mid\mid p$.
\end{lem}

\begin{proof}
Fix some $\mathbb{P}$-name $\dot{z}$ for $z$. We construct the
set $A$ recursively along the ordinals of $M$. More specifically,
we will define via a $\Sigma_{1}$-recursion a sequence $\langle(A^{i},\beta_{i}):i<\OR^{M}\rangle$
and set $A=\bigcup_{i<\OR^{M}}A^{i}$.Let $\beta_{0}$ be the least
$\beta$ such that 
\[
M\mid\beta\models\exists\alpha<\lambda\exists p\in\mathbb{P}(\exists\dot{x}\in M^{\mathbb{P}}(p\Vdash_{M\mid\beta}^{\mathbb{P}}\psi(\dot{x},\check{\alpha},\dot{z}))),
\]
i.e. $\beta$ is such that there is some $\alpha<\lambda$ such that
$W_{\varphi,\check{\alpha},\dot{z}}\cap(M\mid\beta)\neq\emptyset$.
Let $A^{0}$ be such that $(\alpha,p)\in A^{0}$ iff $\alpha<\lambda$,
$W_{\varphi,\check{\alpha},\dot{z}}\cap(M\mid\beta)\neq\emptyset$,
and $p$ is the $<_{M}$-least element in $W_{\varphi,\check{\alpha},\dot{z}}\cap(M\mid\beta)$.
Here, $<_{M}$ denotes the canonical $\Sigma_{1}$-definable well-order
of $M$. Note that $A^{0}$ is a bounded subset of $\delta$. (Actually,
of $\delta^{2}$, but via coding we may assume that $A^{0}\subset\delta$).
Thus, $A^{0}\in M$. Before we continue the construction let us introduce
the following notation: For $\alpha<\lambda$ and $i<\OR^{M}$ let
$A_{\alpha}^{i}:=\{q\in\mathbb{P}:(\alpha,q)\in A^{i}\}$, where $A^{i}$
is to be defined. 

Let $\gamma+1<\OR^{M}$ and suppose that $\langle(A^{i},\beta_{i}):i\leq\gamma\rangle$
is defined such that $\langle(A^{i},\beta_{i}):i\leq\gamma\rangle\in\Sigma_{1}^{M}(\{\lambda,\psi,\dot{z},\mathbb{P}\})$.
We define $\beta_{\gamma+1}$ as the least $\beta$ such that 
\[
M\mid\beta\models\exists\alpha<\lambda\exists p\in\mathbb{P}((\forall q\in A_{\alpha}^{\gamma}\rightarrow p\perp q)\land p\in W_{\varphi,\check{\alpha},\dot{z}}),
\]
if there is such $\beta$. In other words $\beta$ is such that for
some $\alpha$ there is an element $p\in W_{\varphi,\check{\alpha},\dot{z}}\cap(M\mid\beta)$
incompatible with the elements of $A_{\alpha}^{\gamma}$. In case
$\beta_{\gamma+1}$ is undefined we stop the recursion. Note that
in this case $A^{\gamma}\in M$.

In case $\beta_{\gamma+1}$ is defined, we set $A^{\gamma+1}$ to
be the union of $A^{\gamma}$ with the set of all $(\alpha,p)$ such
that $p$ is the $<_{M}$-least $q$ such that 
\[
M\mid\beta_{\gamma+1}\models q\in W_{\varphi,\check{\alpha},\dot{z}}\land\forall r\in A_{\alpha}^{\gamma}\rightarrow r\perp q,
\]
if there is such $q$. Since $A^{\gamma+1}$ is definable over $M\mid\beta_{\gamma+1}$,
$A^{\gamma+1}\in M$. 

Now suppose that $\gamma<\OR^{M}$ is a limit ordinal and $\langle A^{i}:i<\gamma\rangle$
is defined. We set $\beta_{\gamma}=0$ and $A^{\gamma}=\bigcup_{i<\gamma}A^{i}$.
Note that since $\langle A_{i}:i<\gamma\rangle\in\Sigma_{1}^{M}(M)$
by $KP$, $A^{\gamma}\in M$.

We aim to see that the recursive definition of $\langle(A^{i},\beta_{i}):i<\OR^{M}\rangle$
stops before the ordinal height of $\OR^{M}$, i.e. there is some
$\gamma<\OR^{M}$ such that $\beta_{\gamma}$ is not defined. Note
that since $\langle(A^{i},\beta_{i}):i<\OR^{M}\rangle\in\Sigma_{1}^{M}(M)$,
if we assume that the recursion does not stop, $\bar{A}:=\bigcup_{i<\OR^{M}}A^{i}\in\Sigma_{1}^{M}(M)$
is an unbounded subset of $\delta$: Suppose for the sake of contradiction
that $\bar{A}$ is bounded in $\delta$. Then, since $\rho_{1}^{M}=\delta$,
$\bar{A}\in M$. But the definition of $\bar{A}$ gives a cofinal
and total $f:\bar{A}\rightarrow\OR^{M}$ such that $f\in\Sigma_{1}^{M}(M)$.
But by $\Sigma_{1}$-bounding, this means that $\OR^{M}\in M$. Contradiction!

In particular, the following holds in $M$,
\[
\forall\beta<\delta\exists(\alpha,\gamma)(\alpha<\lambda\land\gamma\in(\beta,\delta)\land\exists\beta_{i}(\gamma\text{ is added to }A_{\alpha}\text{ at stage }\beta_{i})).
\]
But by $\Sigma_{1}$-bounding, there is some $\beta_{i}<\OR^{M}$
such that for all $\beta<\delta$, there is some $\gamma\in(\beta,\delta)$
added to some $A_{\alpha}^{i}$ for $\alpha<\lambda$. But this is
inside $M\mid\beta_{i}+\omega$. Thus, as $\delta$ is a regular cardinal
of $M$, there is some $\alpha$ to which unboundedly many $\gamma<\delta$
are added. Contradiction by Lemma \ref{lem:BoundedAntichains}!

Let $\gamma<\OR^{M}$ be least such that $\beta_{\gamma}$ is undefined.
Note that $\gamma$ is by definition a successor ordinal, i.e. $\gamma=\eta+1$
for some $\eta<\OR^{M}$. Set $A:=A^{\eta}$. By construction $A\in M$.
For $\alpha<\lambda$, set $A_{\alpha}:=\{p\in\mathbb{P}:(\alpha,p)\in A\}$.
Since the recursion stops before $\OR^{M}$ we have that for all $\alpha<\lambda$,
$A_{\alpha}\subset W_{\varphi,\check{\alpha},\dot{z}}$ is a maximal
antichain in $W_{\varphi,\check{\alpha},\dot{z}}$. 
\end{proof}
\begin{thm}
\label{thm:KP in delta c.c. forcings}Let $M$ be a $0$-countably
iterable admissible premouse with a largest, regular, and uncountable
cardinal $\delta$ and $\mathbb{P}\in M$ is such that $M\models"\mathbb{P}\text{ has the }\delta\text{-c.c.}"$.
Then $M[g]\models KP$ for any $M$-generic $g\subset\mathbb{P}$.
\end{thm}

\begin{proof}
We will verify that $M[g]$ satisfies $\Sigma_{0}$-bounding and leave
the remaining axioms of $KP$ as an exercise. 

Let $g\subset\mathbb{P}$ be $M$-generic and suppose for the sake
of contradiction that $M[g]\not\models"\Sigma_{0}\text{-bounding}"$.
Let $\psi$ be a $\Sigma_{1}$-formula $\psi$ and $y\in M[g]$ such
that they constitute a counterexample to $\Sigma_{0}$-bounding, i.e.
\[
M[g]\models\forall\alpha<\delta\exists x(\psi(\alpha,x,y)),
\]
but there exists no $z\in M[g]$ such that 
\[
M[g]\models\forall\alpha<\delta\exists x\in z(\psi(\alpha,x,y)).
\]

We distinguish two cases. First, suppose that there is some $\lambda<\delta$
such that there exists no $z\in M[g]$ such that 
\[
M[g]\models\forall\alpha<\lambda\exists x\in z(\psi(\alpha,x,y)).
\]
Then, by the previous Lemma \ref{lem:maximal antichain lemma} there
is $A\in M$ such that for $\alpha<\lambda$, $A_{\alpha}:=\{p\in\mathbb{P}:(\alpha,p)\in A\}$is
a maximal antichain in $W_{\varphi,\check{\alpha},\dot{y}}$, where
$\varphi\equiv\exists\psi$ and $\dot{y}\in M^{\mathbb{P}}$ is a
$\mathbb{P}$-name for $y$. Note that technically $\varphi$ is not
of the form required by Definition \ref{def:set of strongly forcing coniditions},
but this is easy to fix by using Pairing.
\begin{claim*}
For $\alpha<\lambda$, $A_{\alpha}\cap g\neq\emptyset$.
\end{claim*}
\begin{proof}
Fix some $\alpha<\lambda$. Note that since $M[g]\models\exists x\psi(\alpha,x,y)$,
there is some $b\in M[g]$ such that $M[g]\models\psi(\alpha,b,y)$.
By Theorem \ref{thm:Sigma 1 Forcing Theorem} there is some $p\in g\cap W_{\varphi,\check{\alpha},\dot{y}}$.
We claim that this implies that $A_{\alpha}\cap g\neq\emptyset$.
To this end let $D_{\alpha}:=\{p\in\mathbb{P}:\forall q\in A_{\alpha}(p\perp q)\}$.
Note that since $A_{\alpha}\in M$, $D_{\alpha}\in M$. Moreover,
$A_{\alpha}\cup D_{\alpha}\in M$ is pre-dense subset of $\mathbb{P}$
in $M$ and thus, $g\cap(A_{\alpha}\cup D_{\alpha})\neq\emptyset$
by the $M$-genericity of $g$. Suppose for the sake of contradiction
that $g\cap D_{\alpha}\neq\emptyset$ and let $q\in g\cap D_{\alpha}$.
Since $q,p\in g$, there is $r\in g$ such that $r\leq q,p$. However,
this means that $r\in W_{\varphi,\check{\alpha},\dot{y}}$, since
$r\leq p$. But $r$ is incompatible with every element of $A_{\alpha}$,
so $A_{\alpha}$ is not a maximal antichain in $W_{\varphi,\check{\alpha},\dot{y}}$.
Contradiction! Thus, $A_{\alpha}\cap g\neq\emptyset$.
\end{proof}
We have established that for all $\alpha<\lambda$, $A_{\alpha}\cap g\neq\emptyset$.
Moreover, by the previous Lemma we have that 
\[
M\models\forall a\in A\exists\dot{x}(\exists p\in\mathbb{P}\exists\alpha<\lambda(a=(\alpha,p)\land p\Vdash_{M}^{\mathbb{P}}\psi(\check{\alpha},\dot{x},\dot{y})).
\]

This is the antecedence of an instance of the $\Sigma_{1}$-collection
scheme since the part in parentheses is $\Sigma_{1}$ (Note that we
are using here the fact that $\Vdash_{M}^{\mathbb{P}}$ for $\Sigma_{1}$-statements
is $\Sigma_{1}$-definable over $M$.). Thus, there is $z\in M$ such
that 
\[
M\models\forall a\in A\exists\dot{x}\in z(\exists p\in\mathbb{P}\exists\alpha<\lambda(a=(\alpha,p)\land p\Vdash_{M}^{\mathbb{P}}\psi(\check{\alpha},\dot{x},\dot{y}))).
\]
Let $z^{g}:=\{\dot{x}^{g}:x\in z\cap M^{\mathbb{P}}\}$. Note that
$z^{g}\in M[g]$, as $z\cap M^{\mathbb{P}}\in M\cap M^{\mathbb{P}}$.
It follows that 
\[
M[g]\models\forall\alpha<\lambda\exists x\in z^{g}(\psi(\alpha,x,y)),
\]
a contradiction!

Let us now turn towards the second case, i.e. we assume that for all
$\lambda<\delta$ there is $z\in M[g]$ such that 
\[
M[g]\models\forall\alpha<\lambda\exists x\in z(\psi(\alpha,x,y)).
\]
Let us associate to $\psi$ a function $f\in\Sigma_{1}^{M[g]}(M[g])$
such that $f:\delta\rightarrow\OR^{M[g]}$ and for $\alpha<\delta$,
$f(\alpha)$ is the least $\beta$ such that there is $x\in M[g]\mid\beta$
such that $M[g]\models\psi(\alpha,x,y)$. By our assumption $f\restriction\alpha\in M[g]$
for every $\alpha<\delta$, but $f\notin M[g]$. Let us define an
auxiliary function $F$ with domain $\delta$ such that $F(\alpha)=f\restriction\alpha$.
Note that $F\in\Sigma_{1}^{M[g]}(M[g])$, i.e. there is a $\Sigma_{1}$-formula
$\varphi_{F}$ and $z\in M[g]$ such that 
\[
(a,b)\in F\iff M[g]\models\varphi_{F}(a,b,z).
\]

Let 
\[
X:=\{p\in\mathbb{P}:\exists\beta<\OR^{M}(p\Vdash_{M\mid\beta}^{\mathbb{P}}\forall\alpha<\check{\delta}\exists y\varphi_{F}(\alpha,y,\dot{z}))\}.
\]

Note that $X\in\Sigma_{1}^{M}(M)$ and let $\varphi_{X}$ be the defining
formula and $a\in M$ the corresponding parameter.

We now aim to construct in a similar way as in the proof of the previous
Lemma a maximal antichain $A$ in $X$. Let $A_{0}:=\{p\}$ for some
$p\in X$. Suppose that $\langle A_{i}:i<\gamma\rangle$ is defined
via a $\Sigma_{1}$-recursion for some $\gamma<\OR^{M}$. If $\gamma$
is a limit ordinal, let $A_{\gamma}=\bigcup_{i<\gamma}A_{i}$. Then,
$A_{\gamma}\in M$ and $\langle A_{i}:i\leq\gamma\rangle\in\Sigma_{1}^{M}(M)$.
If $\gamma$ is not a limit, let $\beta_{\gamma}<\OR^{M}$ be the
least $\beta$ such that
\[
M\mid\beta\models\exists p\in\mathbb{P}\forall q\in A_{\gamma-1}(\varphi_{X}(p,a)\land q\perp p),
\]
if there exists such $\beta$. In the case that there is no such $\beta$,
stop the recursion. In case $\beta_{\gamma}$ is defined, let $p_{\gamma}$
be the $<_{M}$-least $p$ such that 
\[
M\mid\beta_{\gamma}\models\forall q\in A_{\gamma-1}q\perp p\land\varphi_{X}(p,a).
\]
Let $A_{\gamma}:=A_{\gamma-1}\cup\{p_{\gamma}\}$. Is is not hard
to verify that $\langle A_{i}:i\leq\gamma\rangle\in\Sigma_{1}^{M}(M)$
and $A_{\gamma}\in M$. 

Similar to the proof of the previous Lemma we aim to see that there
is a least $\gamma<\OR^{M}$ such that $\beta_{\gamma}$ is undefined,
which will show that $A:=A_{\gamma-1}$ is a maximal antichain in
$X$ and $A\in M$.

Suppose for the sake of contradiction that for all $\gamma<\OR^{M}$,
$\beta_{\gamma}$ is defined. Let $\bar{A}:=\bigcup_{\gamma<\OR^{M}}A_{\gamma}$.
Clearly, $\bar{A}\in\Sigma_{1}^{M}(M)$. Since $\bar{A}$ is by construction
an antichain, by Lemma \ref{lem:BoundedAntichains} $\bar{A}$ is
bounded in $\delta$. In particular, $\bar{A}\in M$. However, as
in the previous proof, the definition of $\bar{A}$ give rise to a
function $f\in\Sigma_{1}^{M}(M)$ such that $\dom(f)=\bar{A}$ and
$f$ is cofinal in $\OR^{M}$. But this is a contradiction!

Let $D:=\{p\in\mathbb{P}:\forall q\in A(q\perp p)\}$. As $A\in M$,
$D\in M$ and thus $A\cup D\in M$. $A\cup M$ is pre-dense and therefore,
$g\cap(A\cup D)\neq\emptyset$. Note that $g\cap A=\emptyset$, so
$g\cap D\neq\emptyset$.

Let $\tilde{p}\in D\cap g$. Note that this means that there is no
extension $p$ of $\tilde{p}$ such that $p\Vdash_{M}^{w.\mathbb{P}}\forall\alpha<\check{\delta}\exists y\varphi_{F}(\alpha,y,\dot{z})$.
Let $\langle p_{i}:i<\delta\rangle\in M$ be such that $p_{i}\Vdash_{M}^{w.\mathbb{P}}\exists y\varphi_{F}(i,y,\dot{z})$
and $p_{i}\leq\tilde{p}$ for all $i<\delta$. We may find such $\langle p_{i}:i<\delta\rangle$
in $M$, as $\Sigma_{1}$-bounding holds in $M$ and 
\[
M\models\forall i<\delta\exists p(p\Vdash_{M}^{w.\mathbb{P}}\exists y\varphi_{F}(i,y,\dot{z}))\land p\leq\tilde{p}).
\]
Note that for every $i<\delta$ there exists such $p\in\mathbb{P}$,
since $\tilde{p}\in g$ and $M[g]\models\forall i<\delta\exists y\varphi_{F}(i,y,z)$
and so by Theorem \ref{thm:Sigma 1 Forcing Theorem} the existence
of $p$ follows. 

Inside $M$ we are now going to divide $\langle p_{i}:i<\delta\rangle$
into $\delta$-many blocks in the following manner:

Set $\gamma_{0}=0$. For $\xi\in\delta\setminus1$ we define the auxiliary
sequence $\langle r_{i}^{\xi}:i<\delta\rangle$ in the following manner:
Let $r_{0}^{1}=p_{0}$. Suppose $\langle r_{i}^{1}:i<\gamma\rangle$
for $\gamma<\delta$ is defined. If there is $r\leq p_{\gamma}$ such
that for all $\xi<\gamma$, $r\perp p_{\xi}$, let $r_{\gamma}^{1}$
be such $r$. If there is no such $r\leq p_{\gamma}$, let $r_{\gamma}^{1}=\star$.
By the $\delta$-c.c. there is a $\gamma<\delta$ such that for all
$\xi\geq\gamma$, $r_{\xi}^{1}=\star$. Let $\gamma_{1}$ be the least
$\gamma$ such that for all $\xi\geq\gamma$, $r_{\xi}^{1}=\star$. 

Suppose that $\langle\gamma_{j}:j<\xi\rangle$ has been defined for
some $\xi<\delta$. If $\xi$ is a limit let $\gamma_{\xi}=\sup\{\gamma_{j}:j<\xi\}$.
If $\xi=\eta+1$ for some $\eta<\delta$ proceed as follows: Let $r_{0}^{\xi}=p_{\gamma_{\eta}}$.
Suppose that $\langle r_{i}^{\xi}:i<\gamma\rangle$ is defined for
some $\gamma<\delta$. If there is $r\leq p_{\gamma_{\eta}+\gamma}$
such that for all $i\in[\gamma_{\eta},\gamma_{\eta}+\gamma)$, $r\perp p_{i}$,
let $r_{\gamma}^{\xi}$ be such $r$. If there is no such $r$, let
$r_{\gamma}^{\xi}=\star$. As before, by the $\delta$-c.c. there
is $\gamma<\delta$ such that for all $i\geq\gamma$, $r_{i}^{\xi}=\star$.
Let $\gamma_{\xi}$ be the least such $\gamma$. Note that $\langle\gamma_{j}:j<\delta\rangle\in M$.

For $\xi<\delta$, set $D_{\xi}:=\{p\in\mathbb{P}:\exists i\in[\gamma_{\xi},\gamma_{\xi+1})(p\leq p_{i})\}$.
We define the following order on $\{D_{\xi}:\xi<\delta\}$: For $\xi,\xi^{\prime}<\delta$
let $D_{\xi}\leq D_{\xi^{\prime}}$ iff for all $p\in D_{\xi}$ there
exists $q\in D_{\xi^{\prime}}$ such that $p\mid\mid q$. We claim
that for $\xi<\xi^{\prime}<\delta$, $D_{\xi^{\prime}}\leq D_{\xi}$.
Let $r\in D_{\xi^{\prime}}$ and let this be witnessed by $i\in[\gamma_{\xi^{\prime}},\gamma_{\xi^{\prime}+1})$,
i.e. $r\leq p_{i}$. Suppose for the sake of contradiction that for
all $q\in D_{\xi}$, $p\perp q$. In particular, for all $j\in[\gamma_{\xi},\gamma_{\xi+1})$,
$p\perp p_{j}$. But then $r$ witnesses that $i\in[\gamma_{\xi},\gamma_{\xi+1})$.
Contradiction! Note that for $\xi<\xi^{\prime}<\delta$, such that
$D_{\xi^{\prime}}<D_{\xi}$, i.e. $D_{\xi^{\prime}}<D_{\xi}$ and
$D_{\xi}\not\leq D_{\xi^{\prime}}$, there exists $p\in D_{\xi}$
such that for all $q\in D_{\xi^{\prime}},$$p\perp q$. Thus, as $\mathbb{P}$
has the $\delta$-c.c. in $M$ and $(\{D_{\xi}:\xi<\delta\},\leq)\in M$,
there are no properly descending $<$-chains on $\{D_{\xi}:\xi<\delta\}$.
But this means that there is a least $i_{0}<\delta$ such that for
all $\xi,\xi^{\prime}\geq i_{0}$, $D_{\xi}=^{*}D_{\xi^{\prime}}$,
where $D_{\xi}=^{*}D_{\xi^{\prime}}$ means that $D_{\xi}\leq D_{\xi^{\prime}}$
and $D_{\xi}\geq D_{\xi^{\prime}}$. Let $D:=D_{i_{0}}$. Note that
as $D_{\xi}$ is trivially non-empty for every $\xi<\delta$, $D\neq\emptyset$.

Let $q\in D$. We claim that for all $\xi<\delta$ $q\Vdash_{M}^{w.\mathbb{P}}\exists j\in[\dot{\gamma}_{\xi},\dot{\gamma}_{\xi+1})\exists y\varphi_{F}(j,y,\dot{z})$.
By Theorem \ref{thm:Sigma 1 Forcing Theorem} this is equivalent to
$q\Vdash_{M}^{\mathbb{P}}\exists j\in[\dot{\gamma}_{\xi},\dot{\gamma}_{\xi+1})\exists y\varphi_{F}(j,y,\dot{z})$
for all $\xi<\delta$. 

Note that for every $\xi\ge i_{0}$, the set $C_{\xi}:=D_{\xi}\cup\{r\in\mathbb{P}:r\perp q\}$
is dense in $\mathbb{P}$: Let $p\in\mathbb{P}$. If $p\perp q$,
$p\in C_{\xi}$. If $p\mid\mid q$, then there is $r\leq p,q$. But
by definition of $D$, this means that $r\in D$. But then, since
$D=^{*}D_{\xi}$, there is $s\in D_{\xi}$ such that $s\mid\mid r$.
Let this be witnessed by $p^{\prime}$, i.e. $p^{\prime}\leq s,r$.
By the definition of $D_{\xi}$, $p^{\prime}\in D_{\xi}$ and $p^{\prime}\leq p$. 

But this means that for $\xi\geq i_{0}$, $D_{\xi}$ is dense below
$q$. Thus for any $p\leq q$, there is some $p^{\prime}\in D_{\xi}$
such that $p^{\prime}\leq p$. But by definition $p^{\prime}\leq p_{j}$
for some $j\in[\gamma_{\xi},\gamma_{\xi+1})$ and thus, $p^{\prime}\Vdash_{M}^{\mathbb{P}}\exists y\varphi_{F}(\check{j},y,\dot{z})$.
This implies that $p^{\prime}\Vdash_{M}^{\mathbb{P}}\exists j\in[\dot{\gamma}_{\xi},\dot{\gamma}_{\xi+1})\exists y\varphi_{F}(j,y,\dot{z})$,
which is what we wanted. 

Note that by the definition of $F$, if $F(\alpha)$ is defined, then
$F(\beta)$ is defined for all $\beta<\alpha$. Since $\sup\{\gamma_{\xi}:\xi<\delta\}=\delta$,
this implies that for all $j<\delta$, $q\Vdash_{M}^{\mathbb{P}}\exists y\varphi_{F}(\check{j},y,\dot{z})$.
By Theorem \ref{thm:Sigma 1 Forcing Theorem}, there is for every
$j<\delta$ some $\eta_{j}$ and $\dot{y}\in M^{\mathbb{P}}$ such
that $q\Vdash_{M\mid\eta_{j}}^{\mathbb{P}}\varphi_{F}(\dot{j},\dot{y},\dot{z})$.
But then by $\Sigma_{1}$-bounding there is some $\eta<\OR^{M}$ such
that 
\[
M\models\forall j<\delta\exists\dot{y}\in M\mid\eta(q\Vdash_{M\mid\eta}^{\mathbb{P}}\varphi_{F}(\dot{j},\dot{y},\dot{z})).
\]
This means $q\in X$. But since $q\leq\tilde{p}$, $q\perp r$ for
all $r\in A$. Contradiction, since $A$ was supposed to be a maximal
antichain in $X$!
\end{proof}
\begin{rem}
Note that in general it is not the case that if $N$ is admissible
and $g\subset\mathbb{P}$ is $N$-generic for some forcing poset $\mathbb{P}\in N$,
then $N[g]$ is admissible. For this to hold $g$ must meet all dense
open subsets of $\mathbb{P}$ that are unions of a $\boldsymbol{\Sigma_{1}}$
and a $\boldsymbol{\Pi_{1}}$ class over $N$. See \cite{MR3337222}
for more on this. Proposition \ref{prop:counterexample} gives an
explicit example of the generic extension to be admissible.
\end{rem}

From the proof of Theorem \ref{thm:KP in delta c.c. forcings} we
immediately get the following two Corollaries
\begin{cor}
Let $M$ be an admissible premouse with a largest, regular, and uncountable
cardinal $\delta$ such that $\rho_{1}^{M}<\OR^{M}$ and $M$ is $1$-sound.
Let $\mathbb{P}\in M$ is such that $M\models"\mathbb{P}\text{ has the }\delta\text{-c.c.}"$.
Then $M[g]\models KP$ for any $M$-generic $g\subset\mathbb{P}$.
\end{cor}

\begin{cor}
\label{cor:KP preservations extender algebra}Suppose that $M$ is
a\textcolor{red}{{} }$0$-countably iterable admissible premouse with
a largest, regular, and uncountable cardinal $\delta$ and let $\mathbb{B}\in M$
be the extender algebra with $\delta$-many generators as defined
inside $M$. Let $g\subset\mathbb{B}$ be $M$-generic. Then $M[g]\models KP$.
\end{cor}

If we replaced $M\models"\mathbb{P}\text{ has the }\delta\text{-c.c.}"$
in the statement of Theorem \ref{thm:KP in delta c.c. forcings} with
$M\models"\mathbb{P}\text{ is }<\delta\text{-closed}"$, Theorem \ref{thm:KP in delta c.c. forcings}
is false as the following Proposition shows
\begin{lem}
\label{lem:Stacking Mice}Suppose that $M$, $M_{1}$, and $M_{2}$
are sound premice such that $\OR^{M}$ is a regular cardinal (in $V$),
$\rho_{\omega}^{M_{1}}=\rho_{\omega}^{M_{2}}=\OR^{M}$, and Condensation
holds of $M_{1}$ and $M_{2}$. Then, either $M_{1}\unlhd M_{2}$
or $M_{2}\unlhd M_{1}$.
\end{lem}

The Lemma follows from the proof of Lemma 3.1. in \cite{jensen_schimmerling_schindler_steel_2009}.
\begin{prop}
\label{prop:counterexample}Let $M$ be a $1$-sound admissible premouse
with a largest, regular cardinal $\delta$ such that $\rho_{1}^{M}<\OR^{M}$.
Let $\mathbb{C}_{\delta}:=(\delta^{<\delta})^{M}$. Then there is
an $M$-generic $g\subset\mathbb{C}_{\delta}$ such that $M[g]\not\models KP$.
\end{prop}

\begin{proof}
Note that $\mathbb{\mathbb{C}_{\delta}}\in M$ is such that 
\[
M\models"\mathbb{C}_{\delta}\text{ is a }<\delta\text{-closed, splitting, and atom-less forcing}".
\]
By Lemma \ref{lem:1 sound largest cardinal} and $1$-soundness, $\delta=\rho_{1}^{M}$
and $M=H_{1}^{M}(\delta\cup\{p_{1}^{M}\})$. Thus, there is a partial
surjective function $h:\delta\rightarrow M$ such that $h\in\Sigma_{1}^{M}(\{p_{1}^{M}\})$.
Let 
\[
\tilde{D}:=\{\xi<\delta:M\models"h(\xi)\subset\mathbb{C}_{\delta}\text{ is dense in }\mathbb{C}_{\delta}"\}.
\]
Note that $\tilde{D}\notin M$, as otherwise $M$ could construct
an $M$-generic. Let 
\[
D:=\langle\xi_{i}:i<\delta\rangle
\]
 be the monotone enumeration of $\tilde{D}$. (It is easy to see that
there $\delta$-many dense subsets in $M$.) Note, since $h\in\Sigma_{1}^{M}(M)$,
$\tilde{D}\in\Sigma_{1}^{M}(M)$. However, $D\notin\Sigma_{1}^{M}(M)$,
as otherwise by $\Sigma_{1}$-bounding, $D\in M$.

The idea is now to construct an $M$-generic $g$ which codes in a
$\Sigma_{1}$-fashion the set $D$ so that $D\in M[g]$ if $M[g]\models KP$. 

Let us now define $\tilde{g}$ via a $\Sigma_{1}$-recursion. We will
have $\tilde{g}=\{p_{i}:i<\delta\}$ such that $p_{i}\leq p_{j}$
for $j<i$. Set $p_{0}:=\langle\xi_{0}\rangle$ and $p_{1}$ to be
the $<_{M}$-least $p\in h(\xi_{0})$ such that $p\leq p_{0}$. Note
that there is such $p$, since $h(\xi_{0})$ is dense in $\mathbb{P}$.

Suppose that $\langle p_{i}:i<\lambda\rangle$ is defined, where $\lambda<\delta$
is a limit ordinal, such that $p_{i}\leq p_{j}$ for $j\leq i<\lambda$.
Note that $D\restriction\lambda\in M$, since $D\in\Sigma_{1}^{M}(M)$
and $\delta$ is $\Sigma_{1}$-regular by Lemma \ref{lem:Sigma1Regularity}.
But since the construction so far is $\Sigma_{1}$ in the parameter
$D\restriction\lambda$ and $\mathbb{C}_{\delta}$ this implies that
$\langle p_{i}:i<\lambda\rangle\in M$. Thus, we can set $p_{\lambda}=(\bigcup_{i<\lambda}p_{i})^{\frown}\langle\xi_{\lambda}\rangle$.
Let $p_{\lambda+1}$ be the $<_{M}$-least $p\in h(\xi_{\lambda})$
such that $p\leq p_{\lambda}$.

We now turn towards the successor case. Suppose that $\langle p_{i}:i<\eta\rangle$
is defined, where $\eta=\gamma+1$ for some $\gamma<\delta$, which
is a successor. We distinguish whether $\eta$ is odd or even. If
$\eta$ is even, i.e. $\eta=\lambda+2n$ for some $n<\omega$ and
some limit $\lambda\leq\eta$, set $p_{\eta}:=p_{\gamma}^{\frown}\langle\xi_{\lambda+n}\rangle$.
If $\eta$ is odd, let $p_{\eta}$ be the $<_{M}$-least $p\in h(\xi_{\gamma})$
such that $p\leq p_{\gamma}$.

Let $g$ be the upwards-closure of $\tilde{g}$ in $\mathbb{\mathbb{C}_{\delta}}$.
By definition of $\tilde{g}$, $g$ is $M$-generic. Moreover, it
is easy to see that $D$ is $\Sigma_{1}$-definable from $\{\mathbb{C}_{\delta},\rho_{1}^{M},g\}$.
Thus, if $M[g]\models KP$, then $D\in M[g]$. Let $h^{\prime}:\delta\rightarrow\OR^{M}$
be such that $h^{\prime}(\alpha)$ is the least $\beta<\OR^{M}$ such
that $h(\alpha)\in M\mid\beta$, if $h(\alpha)$ is defined. Clearly,
$h^{\prime}\in\Sigma_{1}^{M}(\{p_{1}^{M}\})$. 

Note that since $\mathbb{E}^{M}\restriction\delta\in M[g]$ by the
previous Lemma \ref{lem:Stacking Mice} we can define $\E^{M}$ over
$M[g]$ in a $\Sigma_{1}$-fashion as the extender sequence of the
stack of sound premice extending $M\mid\delta$. Since $M$ is passive,
it follows that $M\in\Sigma_{1}^{M[g]}(M[g])$. Thus, $h^{\prime}\in\Sigma_{1}^{M[g]}(M[g])$.
Moreover, as $\mathbb{P}$ is $<\delta$-closed and atom-less, $\ran(h^{\prime}\restriction D)$
is cofinal in $M$. As $\OR^{M}=\OR^{M[g]}$, $h\restriction D$ is
cofinal in $M[g]$. But now if $M[g]\models KP$, since $D\in M[g]$
, $\OR^{M[g]}\in M[g]$. Contradiction! Thus, $M[g]\not\models KP$.
\end{proof}
Next we deal with the preservation of admissibility in iteration trees.
One instance of this is needed in the proof of Lemma \ref{lem:Second Key Lemma}.
The instance that $N\models KP$ if $N$ is a non-dropping iterate
of an admissible premouse with a largest, regular, and uncountable
cardinal is neither needed for the proof of Lemma \ref{lem:Second Key Lemma},
nor for the construction of the counterexample. However, we include
it since together with the previous Corollary \ref{cor:KP preservations extender algebra}
it gives a nice version of genericity iterations for admissible premice
with a largest, regular, and uncountable cardinals.
\begin{lem}
\label{lem:KP is preseverved by rSigma_3}Suppose that $M$ and $N$
are premice and let $i:M\rightarrow N$ be $r\Sigma_{3}$-elementary.
Then $M\models KP$ iff $N\models KP$.
\end{lem}

The Lemma follows directly from the fact that $KP$ has a $r\Sigma_{3}$-axiomatization
and is therefore preserved under $r\Sigma_{3}$-elementary embeddings.

Thus, by Lemma \ref{lem:KP is preseverved by rSigma_3} $\fu MEn\models KP$
if $n\geq2$ and $M$ is an admissible premouse.

\begin{lem}
\label{lem:General Preservation of KP}Let $k\leq\omega$ and let
$\mathcal{T}$ be a $k$-maximal normal iteration tree on $N$ such
that $\lh(\mathcal{T})=\theta+1$. Let $b:=[0,\theta]_{T}$ be the
main branch of $\T$ and let $\eta=\max(b\cap D_{T})$. Then $\M_{\eta}^{*\T}$
is an admissible premouse with a largest, regular, and uncountable
cardinal iff $\M_{\theta}^{\T}$ is an admissible premouse with a
largest, regular, and uncountable cardinal.
\end{lem}

\begin{rem*}
If $b\cap D_{T}=\emptyset$, then $\max(b\cap D_{T})=0$, so that
$\M_{\eta}^{*\T}=N$.
\end{rem*}
\begin{proof}
Note that in the case that $\deg(\T,\theta)\geq2$, the Lemma holds
by Lemma \ref{lem:KP is preseverved by rSigma_3}. Thus, the only
relevant case is $\deg(\T,\theta)\leq1$.

For simplicity of notation we assume that $k\leq1$ and $D_{T}\cap b=\emptyset$.
The other cases are left as an exercise to the reader.

By standard facts $\M_{\theta}^{\T}$ is a premouse with a largest,
regular, and uncountable cardinal. We have to verify that $\M_{\theta}^{\T}$
is admissible. By Lemma \ref{lem:KP in Premice}, it suffices to show
that $\Sigma_{1}$-bounding holds in $\M_{\theta}^{\T}$. 

We proceed by induction on $\theta$. Suppose first that $\mathcal{T}$
is such that $\lh(\T)=\theta=\zeta+1$. Let $\theta^{*}=\pred_{T}(\theta)$.
We have by the induction hypothesis, that $\M_{\theta^{*}}^{\T}$
is an admissible premouse with a largest, regular, and uncountable
cardinal. 

Let $M:=\M_{\theta^{*}}^{\T}$, $E:=E_{\zeta}^{\T}$, and $\kappa:=\crt(E)$.
By our initial remarks we may assume that $M^{\prime}:=\M_{\theta}^{\T}=\fu MEn$,
where $n\leq1$, i.e. either $\rho_{1}^{M}\leq\kappa<\rho_{0}^{M}$
or $\rho_{2}^{M}\leq\kappa<\rho_{1}^{M}$. However, note that $n=0$,
i.e. $\rho_{1}^{M}\leq\kappa<\rho_{0}^{M}$, cannot be, since then
as $M$ is an admissible premouse with a largest, regular, and uncountable
cardinal, $\kappa$ is the largest cardinal of $M$. But this is impossible
by standard facts. 

Thus, we may assume $n=1$. For the sake of contradiction we suppose
that $M^{\prime}\not\models KP$. By our initial remark $\Sigma_{1}$-bounding
fails in $M^{\prime}$, i.e. there is a $\Sigma_{0}$-formula $\varphi$
and $\eta,p\in M^{\prime}$ such that 
\begin{equation}
M^{\prime}\models\forall\alpha<\xi\exists x\varphi(\alpha,x,p)\label{eq:preservation of bounding}
\end{equation}
but there is no $z\in M^{\prime}$ such that 
\begin{equation}
M^{\prime}\models\forall\alpha<\xi\exists x\in z\varphi(\alpha,x,p).\label{eq:contradicting formula}
\end{equation}

Note that since $M$ is admissible we have by $\Sigma_{1}$-bounding
that for any $m<\omega$, $f\in\Sigma_{1}^{M}(M)\cap^{[\kappa]^{m}}M$
implies that $f\in M$. Thus, $\fu ME0=\fu ME1$. In particular, $i:=i_{\theta^{*}\theta}^{\T}:M\rightarrow M^{\prime}$
is $r\Sigma_{2}$-elementary and $i$ is cofinal in $\OR^{M^{\prime}}$. 

Let $\delta$ be the largest cardinal of $M$ and $\delta^{\prime}=i(\delta)$
be the largest cardinal of $M^{\prime}$. By Lemma \ref{lem:KP in Premice}
we may assume that $\xi$ in the above formula (\ref{eq:preservation of bounding})
is equal to $\delta^{\prime}$.

Let us first suppose that $E$ is finitely generated, i.e. there is
$a\in\fin{\lh(E)}$ such that for every $x\in M^{\prime}$ there is
$f\in^{[\kappa]^{\mid a\mid}}M\cap M$ such that $x=[a,f]_{M}^{E}$.
Fix such $a\in\fin{\lh(E)}$ and let $k=\mid a\mid$. Let $g^{\prime}:\delta\rightarrow M^{\prime}$
be the canonical function derived from $\varphi$ in (\ref{eq:preservation of bounding})
by taking the $<_{M^{\prime}}$-least witness. Note that $g^{\prime}\in\Sigma_{1}^{M^{\prime}}(\{p\})$
as we chose $\varphi$ to be $\Sigma_{0}$. We now aim to derive a
contradiction by bounding $g^{\prime}$ in $M^{\prime}$ by constructing
some $g\in\Sigma_{1}^{M}(M)$ which is bounded in $\OR^{M}$ and corresponds
to $g^{\prime}$ such that the image of the bound is a bound for $g^{\prime}$
in $M^{\prime}$. We construct $g$ in the following way:

By the definition of $\fu ME0$ we have for $\alpha<\delta^{\prime}$
some $f_{\alpha}\in^{[\kappa]^{k}}\delta\cap M$ such that $\alpha=[a,f_{\alpha}]_{M}^{E}$
and $f_{p}\in^{[\kappa]^{k}}M\cap M$ such that $p=[a,f_{p}]_{M}^{E}$
. Since we assumed that $\delta$ is regular in $M$, we may assume
that for $\alpha<\delta$, $f_{\alpha}\in^{[\kappa]^{k}}\delta\cap(M\mid\delta)$.
Moreover, since Los's Theorem holds for $\Sigma_{1}$-formulae we
have by (\ref{eq:preservation of bounding}) for all $\alpha<\delta^{\prime}$,
\[
A_{\alpha}:=\{b\in[\kappa]^{k}:M\models\exists x\varphi(f_{\alpha}(b),x,f_{p}(b))\}\in E.
\]
Note that for any $h\in^{[\kappa]^{k}}\delta\cap(M\mid\delta)$ there
is some $\alpha<\delta^{\prime}$ such that $\alpha=[a,h]_{M}^{E}$.
Thus, 
\[
M\models\forall h\in^{[\kappa]^{k}}\delta\cap(M\mid\delta)\exists A(A\in E\land\forall b\in A\exists x\varphi(h(b),x,f_{p}(b))).
\]
Note that this does make sense as $E\in M$: By standard facts about
$k$-maximal normal iteration trees $E$ is close to $M$ and therefore
in particular, for every $a\in\fin{\lh(E)}$, $E_{a}\in\Sigma_{1}^{M}(M)$.
Since $E$ is finitely generated, this means that $E\in\Sigma_{1}^{M}(M)$.
But since $M$ is admissible, this implies that $E\in M$.

Using $\Sigma_{1}$-bounding, we have for every $h\in^{[\kappa]^{k}}M\cap(M\mid\delta)$
and $A$ such that 

\[
M\models A\in E\land\forall b\in A\exists x\varphi(h(b),x,f_{p}(b))
\]
some $Y\in M$ such that 
\[
M\models A\in E\land\forall b\in A\exists x\in Y\varphi(h(b),x,f_{p}(b)).
\]

Thus, 

\[
M\models\forall h\in^{[\kappa]^{k}}\delta\cap(M\mid\delta)\exists A\exists Y(A\in E\land\forall b\in A\exists x\in Y\varphi(h(b),x,f_{p}(b))).
\]
By another application of $\Sigma_{1}$-bounding this gives us 
\[
M\models\forall h\in^{[\kappa]^{k}}\delta\cap(M\mid\delta)\exists A\in(M\mid\beta)\exists Y\in(M\mid\beta)
\]
\[
(A\in E\land\forall b\in A\exists x\in Y\varphi(h(b),x,f_{p}(b)))
\]
for some $\beta<\OR^{M}$. Thus, $g^{\prime}\in J(M^{\prime}\mid i(\beta))$.

Now let us consider the case that $E$ is not finitely generated.
In this case $M^{\prime}$ is the direct limit of 
\[
(\fu M{E_{a}}0,\pi_{ab}:a,b\in\fin{\nu(E)}\land a\subset b),
\]
where $\pi_{ab}$ is the canonical factor-embedding. Note that for
$a\in\fin{\nu(E)}$, $\pi_{0a}$ is as in the previous case a cofinal
and $r\Sigma_{2}$-elementary embedding. Moreover, since $E_{a}$
is finitely generated we have by what we have shown so-far that $\fu M{E_{a}}0\models KP$.
Note that $\pi_{ab}$ factors $\pi_{0b}$ through $\pi_{0a}$ and
the last two embeddings are cofinal and $r\Sigma_{2}$-elementary,
thus, $\pi_{ab}$ is cofinal. It follows from standard facts that
$\pi_{ab}$ is $r\Sigma_{2}$-elementary. For $a\in\fin{\nu(E)}$
let $X_{a}:=\pi_{b\infty}[\fu M{E_{a}}0]$. By Los's Theorem it follows
that $\pi_{a\infty}:\fu M{E_{a}}0\rightarrow M^{\prime}$ is $\Sigma_{1}$-elementary.

Suppose now that $f\in\Sigma_{1}^{M^{\prime}}(M^{\prime})$. By the
definition of a direct limit, there is some $a\in\fin{\nu(E)}$ such
that $f\in X_{a}$. By the previous claim and the fact that $\fu M{E_{b}}0\models KP$
there is some $Z\in X_{a}$ such that $f\in Z$. But that means that
$f\in M^{\prime}$. Thus, by Lemma \ref{lem:KP in Premice}, $M^{\prime}\models KP$.

The case that $\lh(\T)=\theta$, is a limit ordinal, and there is
no drop in model or degree on $b:=[0,\theta)_{T}$ is similar to the
case that $E$ does not have finitely many generators in the successor
case.

Now suppose that $\M_{\theta}^{\T}$ is admissible. Again we argue
by induction on $\theta$. Suppose first that $\theta=\zeta+1$ and
let $\theta^{*}=\pred(\theta)_{T}$. If $\theta^{*}\in D_{T}$, then
we may assume that $\M_{\theta^{*}}^{*\T}$ has a largest, regular,
and uncountable cardinal $\delta$. In the case that $\theta^{*}\notin D_{T}$,
we have by the induction hypothesis that $\M_{\theta^{*}}^{\T}$ has
a largest, regular, and uncountable cardinal $\delta$. Depending
the case suppose for the sake of contradiction that either $\M_{\theta^{*}}^{*\T}$
or $\M_{\theta^{*}}^{\T}$ is not admissible, i.e. $\Sigma_{1}$-bounding
fails. Let $M:=\M_{\theta^{*}}^{*\T}$ or $M:=\M_{\theta^{*}}^{\T}$
depending on the case and let $\eta_{\theta^{*}}$ be the least failure
of $\Sigma_{1}$-bounding, i.e. there is a $\Sigma_{1}$-formula $\varphi$
and $p\in M$ such that $M\models\forall\alpha<\eta_{\theta^{*}}\exists x\varphi(x,\alpha,p)$
but there is no $z\in M$ such that $M\models\forall\alpha<\eta_{\theta^{*}}\exists x\in z\varphi(x,\alpha,p)$.
As before we may assume that $M^{\prime}:=\M_{\theta}^{\T}=\fu MEn$
for $n\leq1$, where $E=E_{\zeta}^{\T}$, since $KP$ has a $r\Sigma_{3}$-axiomatization.
Let $i:=i_{\theta^{*}\theta}^{\T}:M\rightarrow M^{\prime}$ be the
tree embedding and let $\delta^{\prime}=i(\delta)$ be the largest
cardinal of $M^{\prime}$. 

Case 1: $n=1$. If we could show that $M^{\prime}\models\forall\alpha<i(\eta_{\theta^{*}})\exists x\varphi(x,\alpha,i(p))$,
then by the admissibility of $M^{\prime}$, $M^{\prime}\models\exists z^{\prime}\forall\alpha<i(\eta_{\theta^{*}})\exists x\in z^{\prime}\varphi(x,\alpha,i(p))$.
This is a $\Sigma_{1}$-statement, so since $i$ is at least $r\Sigma_{1}$-elementary,
$M\models\exists z\forall\alpha<\eta_{\theta^{*}}\exists x\in z\varphi(x,\alpha,p)$,
which would be a contradiction. 

So let us show that $M^{\prime}\models\forall\alpha<i(\eta_{\theta^{*}})\exists x\varphi(x,\alpha,i(p))$.
Suppose that this is not the case, i.e. there is $\alpha<i(\eta_{\theta^{*}})$
such that 
\[
M^{\prime}\models\forall x\neg\varphi(x,\alpha,i(p)).
\]
This is a $\Pi_{1}$-sentence. Let $a\in\fin{\kappa}$ and $f\in^{[\kappa]^{\mid a\mid}}\OR^{M}\cap\Sigma_{1}^{M}(M)$
such that $\alpha=[a,f]_{E}^{M}$, where $\kappa:=\crt(E_{\zeta}^{\T})$.
Since we have \L o\'{s}-Theorem for $\Sigma_{1}$-formulae, 
\[
\{\vec{\beta}\in[\kappa]^{\mid a\mid}:M\models\forall x\neg\varphi(x,f(\vec{\beta}),p)\}\in E_{a}.
\]
But since $\alpha<i(\eta_{\theta^{*}})$ we have that for a.e. $\vec{\beta}\in\fin\kappa$,
$f(\vec{\beta})<\eta_{\theta^{*}}$. Contradiction!

Case 2: $n=0$. Note that in this case we do not have \L o\'{s}-Theorem
for $\Sigma_{1}$-formulae. Let $\eta_{\theta^{*}}^{\prime}:=\sup(i[\eta_{\theta^{*}}])$
and note that $\eta_{\theta^{*}}^{\prime}$ is a limit ordinal. Suppose
for the sake of contradiction that there is $\alpha^{\prime}<\eta_{\theta^{*}}^{\prime}$
such that 
\[
M^{\prime}\models\forall x\neg\varphi(x,\alpha^{\prime},i(p)).
\]
Let $\alpha<\eta_{\theta^{*}}$ be such that $i(\alpha)>\alpha^{\prime}$.
Note that since $\eta_{\theta^{*}}$ is the minimal failure of $\Sigma_{1}$-bounding
in $M$, 
\[
M\models\exists z\forall\beta<\alpha\exists x\in z\varphi(x,\beta,p).
\]
But this is a $\Sigma_{1}$-statement, so that 
\[
M^{\prime}\models\exists z\forall\beta<i(\alpha)\exists x\in z\varphi(x,\beta,i(p)).
\]
But then in particular, 
\[
M^{\prime}\models\exists x\varphi(x,\alpha^{\prime},i(p)).
\]
Contradiction! Therefore, by $\Sigma_{1}$-bounding, there is $\beta^{\prime}<\OR^{M^{\prime}}$
such that $M^{\prime}\models\forall\alpha<\eta_{\theta^{*}}^{\prime}\exists x\in(M^{\prime}\mid\beta^{\prime})\varphi(x,\alpha,i(p)).$
Since $i$ is cofinal, we may assume without loss of generality that
$\beta^{\prime}\in\ran(i)$. Let $\beta<\OR^{M}$ be such that $i(\beta)=\beta^{\prime}$.
We claim that 
\[
M\models\forall\alpha<\eta_{\theta^{*}}\exists x\in(M\mid\beta)\varphi(x,\alpha,p),
\]
which would be a contradiction! So suppose that there is $\alpha<\eta_{\theta^{*}}$
such that 
\[
M\models\forall x\in(M\mid\beta)\neg\varphi(x,\alpha,p).
\]
This is a $\Pi_{1}$-statement, so that 
\[
M^{\prime}\models\forall x\in(M^{\prime}\mid\beta^{\prime})\neg\varphi(x,\alpha,i(p)).
\]
Contradiction!

Now let us suppose that $\theta$ is a limit ordinal. In the case
that there is some $\gamma\in b\cap\theta$ such that for all $\xi\in(\gamma,\theta)\cap b$,
$i_{\gamma\xi}^{\T}$ is $r\Sigma_{2}$-elementary, we can argue as
in Case 1 of the successor case. If otherwise we can use the argument
from Case 2 of the successor case, since $i_{\gamma\theta}^{\T}$
will be cofinal.
\end{proof}
If we would not require $\delta$ to be regular in the statement of
Lemma \ref{lem:General Preservation of KP}, the Lemma is provably
false as the following example shows
\begin{example}
Let $M$ be an $1$-sound premouse such that $M\models KP$. Suppose
that $M$ has a largest cardinal $\delta>\omega$ such that for some
$\kappa<\delta$, $\cof^{M}(\delta)=\kappa$ and there is a total
$E\in\E^{M}$ such that $\crt(E)=\kappa$. Let $\pi:M\rightarrow\fu ME1$
be the $1$-ultrapower of $M$ via $E$.

Then $\pi$ is discontinuous at $\delta$ so that $\pi(\delta)>\sup(\pi[\delta])$.
However, since $\pi$ is a $1$-embedding, $\rho_{1}^{M}=\sup(\pi[\delta])$.
But by the $r\Sigma_{1}$-elementarity and cofinality of $\pi$, $\pi(\delta)$
is the largest cardinal of $\fu ME1$. If $\fu ME1\models KP$, this
is a contradiction by Lemma \ref{lem:Projectum of KP Structures}!
\end{example}

In this Subsection we have established the following
\begin{thm}
Let $k\leq\omega$ and suppose that $M$ is a $k$-sound, $(k,\mid M\mid^{+}+1)$-iterable
admissible premouse with a largest, regular, and uncountable cardinal
$\delta$. Let $X\subset\delta$. Then there is a successor-length
$k$-maximal iteration tree $\T$ on $M$ such that if $\M_{\infty}^{\T}$
is its last model, $X$ is generic over $\M_{\infty}^{\T}$, and $\M_{\infty}^{\T}[X]\models KP$.
\end{thm}

\subsection{The Construction of the Counterexample\label{subsec:The-Construction-of the counterexample}}

Let $M:=M_{\Sigma_{1}-1}^{\ad}$ and let $N\lhd M$ be as given by
Lemma \ref{lem:N existence lemma}, i.e. the following hold:
\begin{itemize}
\item $M$ has a largest cardinal $\delta$ which is $\Sigma_{1}$-Woodin
in $M$,
\item $N$ is an admissible mouse with greatest cardinal $\kappa$ which
is Woodin in $N$, but not $\Sigma_{1}$-Woodin in $N$, 
\item $\OR^{N}<\kappa^{+M}$, and
\item $\kappa$ is measurable in $M$, as witnessed by an extender $E$
from the extender sequence of $M$ with one generator.
\end{itemize}
The construction of the counterexample is now as follows: We fix some
uncountable, regular cardinal $\Omega\in(\kappa^{++M},\delta)$ of
$M$. Working inside $M$ we will linearly iterate $M\mid\lh(E)$
$\Omega$-many times via $E$ and its images. Let $\T$ be the corresponding
iteration tree on $M\mid\lh(E)$ of length $\lh(\T)=\Omega+1$. Note
that $\T\in M$. It is straightforward to check that $\T$ is non-dropping.
Let $i^{\T}:M\mid\lh(E)\rightarrow\M_{\Omega}^{\T}$ be the tree embedding.
We will have the following
\begin{itemize}
\item $i^{\T}(\kappa)=\Omega$, 
\item $\sup(i^{\T}[\kappa^{+M}])=i^{\T}(\kappa^{+M})=\Omega^{+\M_{\Omega}^{\T}}<\Omega^{+M}$, 
\item $W:=i^{\T}(N)\models KP$ and $\OR^{W}<\Omega^{+\M_{\Omega}^{\T}}$,
and
\item $W$ is an admissible mouse with greatest cardinal $\Omega$ which
is Woodin in $W$, but not $\Sigma_{1}$-Woodin in $W$. In particular,
$\Omega$ is the only Woodin of $W$.
\end{itemize}
Since $\Omega^{+\M_{\Omega}^{\T}}<\Omega^{+M}$, there is, inside
$M$, some $A\subset\Omega$ such that $\otp(A)=\Omega^{+\M_{\Omega}^{\T}}$.
Let $\U$ be the genericity iteration on $W$ for the extender algebra
with $\Omega$-many generators in $W$, which makes $A$ generic over
$\M_{\infty}^{\U}$. Note that $\U\in M$, since by standard arguments
there is for every limit $\alpha\leq\Omega$ a $Q$-structure fore
$\U\restriction\alpha$ in $M$.. Moreover, since $\Omega$ is a limit
cardinal in $\M_{\Omega}^{\T}$, by the usual arguments $\U$ is non-dropping,
$\lh(\U)=\Omega+1$, and $i_{0\Omega}^{\U}(\Omega)=\Omega$. 

Inside $M$ we construct a club $C\subset\Omega$ such that for all
$\alpha\in C$: 
\begin{itemize}
\item there exists some $\pi_{\alpha}:M_{\alpha}\cong X_{\alpha}\prec_{1000}M\mid\mid\Omega^{++}$
such that $M\mid\lh(E)+\omega\cup\{\Omega,A\}\subset\ran(\pi_{\alpha})$,
$\crt(\pi_{\alpha})=\alpha$, and $\pi_{\alpha}(\alpha)=\Omega$,
\item $\alpha=\sup\{\lh(E_{\beta}^{\U}):\beta<\alpha\}$, 
\item $\alpha\in b:=[0,\Omega]_{U}$, and 
\item $\alpha=\crt(E_{\alpha}^{\T})$.
\end{itemize}
The construction of such $C$ is fairly standard, so we will omit
it. Note that for $\alpha\in C$, $\crt(i_{\alpha\Omega}^{\U})\geq\alpha$,
since $\alpha=\sup\{\lh(E_{\beta}^{\U}):\beta<\alpha\}$. Thus, by
the usual argument $\alpha^{+\M_{\alpha}^{\U}}=\alpha^{+\M_{\Omega}^{\U}}$.
\begin{claim*}
For $\alpha\in C$, $\alpha^{+\M_{\alpha}^{\U}}>\alpha^{+\M_{\Omega}^{\T}}$. 
\end{claim*}
\begin{proof}
Fix $\alpha\in C$. Let $(\bar{\T},\bar{\U},\bar{A},\bar{W})\in M_{\alpha}$
such that $\pi_{\alpha}((\bar{\T},\bar{\U},\bar{A},\bar{W}))=(\T,\U,A,W)$.
Note that there are such $\bar{\T},\bar{\U},\bar{W}\in M_{\alpha}$,
since $\T,\U,W$ are definable from $M\mid\lh(E)$, $A$, and $\Omega$.

We will have that $\M_{\alpha}^{\T}=\M_{\alpha}^{\bar{\T}}$ and $\crt(i_{\alpha\Omega}^{\T})=\alpha$.
Thus, $\alpha^{+\M_{\Omega}^{\T}}=\alpha^{+\M_{\alpha}^{\bar{\T}}}$.
By elementarity $\otp(\bar{A})=\alpha^{+\M_{\alpha}^{\bar{\T}}}$.
Moreover, $\bar{\U}$ is a genericity iteration of $\bar{W}$ which
makes $\bar{A}$ generic. Since $\bar{W}$ is the minimal initial
segment of $\M_{\alpha}^{\bar{\T}}$ which models $KP$ and has $\alpha$
as a Woodin cardinal, $\bar{W}=J_{\beta}(\M_{\Omega}^{\T}\mid\alpha)$
for some $\beta\in\OR$. Since $\M_{\Omega}^{\T}$ has no initial
segment which models $KP$ and has a $\Sigma_{1}$-Woodin, $\alpha$
is not Woodin in $J_{\beta+1}(\M_{\Omega}^{\T}\mid\alpha)$. Thus,
$\bar{W}\lhd W$. We also have that $\bar{\U}"="\U\restriction\alpha$,
in the sense that for $\beta<\alpha$, $E_{\beta}^{\U}=E_{\beta}^{\bar{\U}}$.
But this means that $\M_{\alpha}^{\bar{\U}}\lhd\M_{\alpha}^{\U}$. 

Note that $\bar{A}$ is generic over $\M_{\alpha}^{\bar{\U}}$. By
Corollary \ref{cor:KP preservations extender algebra}, $\M_{\alpha}^{\bar{\U}}[\bar{A}]\models KP$.
Thus, $\otp(\bar{A})=\alpha^{+\M_{\alpha}^{\T}}\in\M_{\alpha}^{\bar{\U}}[\bar{A}]$.
But this means that $\OR^{\M_{\alpha}^{\bar{\U}}}=\OR^{\M_{\alpha}^{\bar{\U}}[\bar{A}]}>\alpha^{+\M_{\alpha}^{\T}}$.
Since $\M_{\alpha}^{\bar{\U}}\lhd\M_{\alpha}^{\U}$, this means that
$\alpha^{+\M_{\alpha}^{\U}}>\alpha^{+\M_{\Omega}^{\T}}.$
\end{proof}
We have shown that for every $\alpha\in C$, $\alpha^{+\M_{\Omega}^{\T}}<\alpha^{+\M_{\Omega}^{\U}}$.
If we consider the tree $\U^{*}$ which is just the tree $\U$ considered
not on $W$ but on $W\mid\Omega$, (which we can do, since $\lh(E_{\alpha}^{\U})<\Omega$
for all $\alpha<\Omega$ and $\Omega$ is a regular cardinal, so that
the tree structure does not change) we will have that for all $\alpha\in C$,
$\alpha^{+W\mid\Omega}<\alpha^{+\M_{\Omega}^{\U^{*}}}$. Moreover,
since $i^{\U}(\Omega)=\Omega$, $\OR^{\M_{\Omega}^{\U^{*}}}=\Omega$,
so that $\M_{\Omega}^{\U^{*}}$ is a weasel in the sense of $M$.

However, if $(\T^{\prime},\U^{\prime})$ is the coiteration of $(W\mid\Omega,\M_{\Omega}^{\U^{*}})$,
then $\T^{\prime}=\U^{*}$ and $\U^{\prime}$ is the trivial tree.
Thus, $W\mid\Omega=^{*}\M_{\Omega}^{\U^{*}}$ and there exists a club
$C\subset\Omega$ such that for all $\alpha\in C$, $\alpha^{+W\mid\Omega}<\alpha^{+\M_{\Omega}^{\U^{*}}}$.
This contradicts the conjecture inside $M$.
\begin{rem*}
The construction above also works if we picked $\Omega$ to be $\delta$.
However, in this case we have to modify the construction slightly:
Let $\eta>\kappa$ be a measurable cardinal of $M$ and let $F\in\E^{M}$
be the measure witnessing this. Let $\T^{\prime}$ the linear iteration
of $M\mid\eta^{++M}$ via $F$ and its images of length $\delta+1$.
Note that by $\Sigma_{1}$-bounding $\T^{\prime}\in M$ and that $\delta^{+}$
exists in $\M_{\delta}^{\T^{\prime}}$. Now as before we let $\T$
be a linear iteration of $M\mid\lh(E)$ of via $E$ and its images
of length $\delta+1$. However, note that $\T\in\M_{\delta}^{\T^{\prime}}$.
Thus, there is $A\in\M_{\delta}^{\T^{\prime}}$ such that $\otp(A)=\delta^{+\M_{\delta}^{\T}}$
and $A\subset\delta$. Since $\delta^{+\M_{\delta}^{\T^{\prime}}}$
exists and there is no initial segment of $\M_{\delta}^{\T^{\prime}}$
which models $KP$ and has a $\Sigma_{1}$-Woodin cardinal, we see
by the same argument as before, that $Q$-structures exist for $\U$,
where $\U$ is the genericity iteration of $i_{0\delta}^{\T}(N)$
making $A$ generic.
\end{rem*}

\section{On another Question from CMIP\label{sec:On-another-Conjecture}}

In this last Section we discuss another open question from \cite{CMIP}
concerning the $S$-hull property. Throughout this section $\Omega$
is a fixed measurable cardinal and $\mu_{0}$ is a fixed normal measure
on $\Omega$.

Note that our definitions of thickness and the hull property are different
from the ones in \cite{CMIP}, yet equivalent. We chose these different
definitions in order to emphasize that thickness is a property independent
of a specific weasel.
\begin{defn}
Let $\W$ be weasel and $S\subset\Omega$ be stationary. We say that
$S$ is good for $\W$ iff there is a club $C\subset\Omega$ such
that for all $\alpha\in C\cap S$
\begin{itemize}
\item $\alpha$ is inaccessible,
\item $\alpha^{+}=(\alpha^{+})^{\W}$, and 
\item $\alpha$ is not the critical point of a total-on-$\W$ extender from
the extender sequence of $\W$.
\end{itemize}
\end{defn}

\begin{defn}
Let $\Gamma\subset\Omega$ and $S$ be stationary. We say that $\Gamma$
is $S$-thick iff there is a club $C\subset\Omega$ such that for
all $\alpha\in C\cap S$, $\Gamma\cap\alpha^{+}$ contains an $\alpha$-club
and $\alpha\in\Gamma$. 
\end{defn}

\begin{defn}
Let $\W$ be a weasel and $S\subset\Omega$ be stationary such that
$S$ is good for $\W$. We say that $\W$ has the $S$-hull property
at $\alpha<\Omega$ iff for all $\Gamma\subset\Omega$ which are $S$-thick
\[
\mathcal{P}(\alpha)\cap\W\subset\text{transitive collapse of }\Hull\W{\omega}(\alpha\cup\Gamma).
\]
\end{defn}

In \cite{CMIP} it is proven in Lemma 4.6. on p. 32 that for an $\Omega+1$-iterable
weasel $\W$ for $\mu_{0}$-a.e. $\alpha<\Omega$ the $S$-hull property
holds at $\alpha$.

However, it is mentioned in the paragraph preceding Lemma 4.6. that
it remains open whether the set $HP^{\W}:=\{\alpha<\Omega:\W\text{ has the }S\text{-hull property at }\alpha\}$
is closed. Note that $HP^{\W}$ clearly cannot be closed in the usual
sense, as the following example from \cite{CMIP} p.29 shows: Suppose
that $M$ is an $\Omega+1$-iterable weasel which has the $S$-hull
property at all $\alpha<\Omega$ and there is a total-on-$M$ $E\in\E^{M}$
with at least two generators. Then, by standard arguments $\fu ME\omega$
has the $S$-hull property at all $\alpha<(\crt(E)^{+})^{M}$ but
not at $(\crt(E)^{+})^{M}$. However, $\fu ME\omega$ is still $\Omega+1$-iterable.
Thus, $HP^{\W}$ cannot be closed in the classical sense. 
\begin{defn}
We say that $X\subset\Omega$ is almost closed if for every $\delta\in X$
such that $\delta$ is the supremum of elements of $X$ and elements
of $\Omega\setminus X$, then $\delta\in X$.
\end{defn}

Thus, the question from \cite{CMIP} translates into the following:
Let $\W$ be an $\Omega+1$-iterable weasel and $S\subset\Omega$
stationary such that $S$ is good for $\W$, is the set $HP^{\W}$
almost closed?

This question is anwered positively by the following theorem
\begin{thm}
\label{thm:Hull Property Set is Closed}Let $\W$ be an $\Omega+1$-iterable
weasel and $S\subset\Omega$ be stationary such that $S$ is good
for $\W$. Then the set $HP^{\W}$ is almost closed.
\end{thm}

\begin{proof}
By Lemma 4.5. of \cite{CMIP} there is an $\Omega+1$-iterable weasel
$M$ and an elementary embedding $\pi:M\rightarrow\W$ such that $\ran(\pi)$
is $S$-thick and $M$ has the $S$-hull property at all $\alpha<\Omega$,
i.e. $HP^{M}=\Omega$. Let $(\T,\U)$ be the coiteration of $\W$
and $M$. We will prove the Lemma assuming that $\lh(\T)=\lh(\U)=\Omega+1$
and leave the remaining cases as an excercise to the reader. Since
$\W$ and $M$ are universal, $i^{\T}:\W\rightarrow\M_{\Omega}^{\T}$
and $i^{\U}:M\rightarrow\M_{\Omega}^{\U}$ exist and $\M_{\Omega}^{\T}=\M_{\Omega}^{\U}=:M_{\infty}$.

Note that by the Remark following Example 4.3. in \cite{CMIP} on
p. 29, $M_{\infty}$ has the $S$-hull property at $\xi$ iff for
no $\beta+1\in[0,\Omega]_{U}$, $\xi\in[\crt(E_{\beta}^{\U})^{+\M_{\beta}^{\U}},\nu(E_{\beta}^{\U}))$.
Thus, the set $HP^{M_{\infty}}$ is almost closed. Moreover, by arguments
from the proof of Lemma 4.6 in \cite{CMIP} (using that the set of
fixed points of $i^{\T}$ is $S$-thick) we have that for all $\alpha<\Omega$,
$\W$ has the $S$-hull property at $\alpha$ iff $M_{\infty}$ has
the $S$-hull property at $i^{\T}(\alpha)$. 

Suppose for the sake of contradiction that $HP^{\W}$ is not almost
closed. Let $\delta<\Omega$ be a witness for this. Since $HP^{M_{\infty}}$
is almost closed this means that $i^{\T}(\delta)>\sup(i^{\T}[\delta])$. 

There are two ways it can happen that $i^{\T}(\delta)>\sup(i^{\T}[\delta])$. 

Case 1: There is $\beta+1\in[0,\Omega]_{T}$ such that $\crt(E_{\beta}^{\T})=i_{0\gamma}^{\T}(\delta)$,
where $\gamma=\pred_{T}(\beta+1)$. Note that in this case we must
have that $\sup(i_{0\gamma}^{\T}[\delta])=i_{0\gamma}^{\T}(\delta)$
and $\delta^{\prime}:=i_{0\gamma}^{\T}(\delta)\in HP^{M_{\infty}}.$
Since $\crt(i_{\beta+1\Omega}^{\T})>\delta^{\prime}$, $\M_{\beta+1}^{\T}$
has the $S$-hull property at $\delta^{\prime}$. But then, since
$\M_{\beta+1}^{\T}$ has the $S$-hull property at $\delta^{\prime}$
and $\Pow{\delta^{\prime}}\cap\M_{\gamma}^{\T}=\Pow{\delta^{\prime}}\cap\M_{\beta+1}^{\T}$,
$\M_{\gamma}^{\T}$ has the $S$-hull property at $\delta^{\prime}$.
Contradiction!

Case 2: There is a minimal $\alpha\in[0,\Omega]_{T}$ such that in
$\M_{\alpha}^{\T}$, $i_{0\alpha}^{\T}(\delta)$ is singular and $\cof(i_{0\alpha}^{\T}(\delta))=\kappa$,
where $\kappa=\crt(E_{\gamma}^{\T})$, $\pred_{T}(\gamma+1)=\alpha$,
and $\gamma+1\in[0,\Omega]_{T}$. Let us set $W^{\prime}:=\M_{\alpha}^{\T}$,
$W^{\prime\prime}:=\M_{\gamma+1}^{\T}$, $\delta^{\prime}:=i_{0\alpha}^{\T}(\delta)$,
$j:=i_{\alpha\gamma+1}^{\T}$, and $\delta^{\prime\prime}:=\sup(j[\delta^{\prime}])$.
Note that $\delta^{\prime\prime}<j(\delta^{\prime})$. Moreover, letting
$i:=i_{0\alpha}^{\T}$, we have that $\delta^{\prime}=i(\delta)=\sup(i[\delta])$. 

Note that $\delta^{\prime\prime}$ is a limit of $HP^{W^{\prime\prime}}$
and $\Omega\setminus HP^{W^{\prime\prime}}$. We aim to see that $i_{\gamma+1\Omega}^{\T}(\delta^{\prime\prime})$
is also a limit of $HP^{\M_{\Omega}^{\T}}$ and $\Omega\setminus HP^{\M_{\Omega}^{\T}}$,
since then the $S$-hull property holds at $\delta^{\prime\prime}$
in $W^{\prime\prime}$. For this it suffices to see that $i_{\gamma+1\Omega}^{\T}(\delta^{\prime\prime})=\sup i_{\gamma+1\Omega}^{\T}[\delta^{\prime\prime}]$.
To this end note that $\cof(\delta^{\prime\prime})^{W^{\prime\prime}}=\kappa$
and that all extenders used along $[0,\Omega]_{T}$ after $E_{\gamma}^{\T}$
have critical points greater than $\kappa$. Thus, $i_{\gamma+1\Omega}^{\T}$
is continuous at $\delta^{\prime\prime}$ and the $S$-hull property
holds at $\delta^{\prime\prime}$ in $W^{\prime\prime}$.

We claim that this implies that the $S$-hull property holds at $\delta^{\prime}$
in $W^{\prime}$, which would be a contradiction. Let $A\subset\delta^{\prime}$
such that $A\in W^{\prime}$ and $\Gamma$ be a $S$-thick set. We
need to show that there is a term $\tau$, $\vec{\xi}\in\fin\Gamma$,
and $\vec{\beta}\in\fin{\delta^{\prime}}$ such that 
\[
A=\tau^{W^{\prime}}[\vec{\xi},\vec{\beta}]\cap\delta^{\prime}.
\]
Note that $j(A)\cap\delta^{\prime\prime}\in W^{\prime\prime}$. Thus,
there is a term $\sigma$, $\vec{\zeta}\in\fin\Gamma$, and $\vec{\gamma}\in\delta^{\prime\prime}$
such that 
\[
j(A)\cap\delta^{\prime\prime}=\sigma^{W^{\prime\prime}}[\vec{\zeta},\vec{\gamma}]\cap\delta^{\prime\prime}.
\]
We may assume that $\vec{\zeta}$ is fixed by $j$. Let $\vec{\gamma}=([a_{0},f_{0}]_{E}^{W^{\prime}},...,[a_{n},f_{n}]_{E}^{W^{\prime}})$
such that for $k\leq n$, $a_{k}\in\fin{\nu(E)}$ and $f_{k}\in^{\fin\kappa}W^{\prime}\cap W^{\prime}$.
Note that \L o\'{s}'s Theorem holds, in particular for all $\xi<\delta^{\prime}$
we have that 
\[
j(\xi)\in\sigma^{W^{\prime\prime}}[\vec{\zeta},\vec{\gamma}]\iff
\]
\begin{equation}
\{b\in\fin\kappa:W^{\prime}\models\xi\in\sigma^{W}[\vec{\zeta},f_{0}^{a_{0},a}(b),...,f_{n}^{a_{n},a}(b)]\}\in E_{a},\label{eq:term}
\end{equation}
where $a=\bigcup_{i\leq n}a_{n}$. Furthermore, for $\xi<\delta^{\prime}$,
\begin{equation}
j(\xi)\in\sigma^{W^{\prime\prime}}[\vec{\zeta},\vec{\gamma}]\iff j(\xi)\in j(A)\iff\xi\in A.\label{eq:equivalence}
\end{equation}
Note that for every $c\in\fin{\nu(E)}$, $E_{c}$ is close to $W^{\prime}$.
In particular, since $W^{\prime}\models ZFC$, $E_{c}\in W^{\prime}$
for every $c\in\fin{\nu(E)}$. Moreover, $\delta^{\prime}$ is a limit
cardinal in $W^{\prime}$ and $GCH$ holds in $W^{\prime}$. Thus,
for every $c\in\fin{\nu(E)}$ the ordinal of the extender $E_{c}$
in the $W^{\prime}$-order $<_{W^{\prime}}$ is an ordinal less than
$\delta^{\prime}$. Furthermore, as the ordinals below $\delta^{\prime\prime}$
might be represented via bounded functions in $^{[\kappa]^{<\omega}}\delta^{\prime}$,
we may assume that for $k\leq n$, $f_{k}$ is bounded in $\delta^{\prime}$
and thus again their ordinals in the $W^{\prime}$-order ale less
than $\delta^{\prime}$. But this means that \ref{eq:term} and \ref{eq:equivalence}
give us a term $\tau$ and $\vec{\beta}\in\fin{\delta^{\prime}}$
such that 
\[
A=\tau^{W^{\prime}}[\vec{\zeta},\vec{\beta}]\cap\delta^{\prime}.
\]
\end{proof}
\bibliographystyle{plain}
\bibliography{conjecture}

\end{document}